\documentclass[11pt, english, draft]{article}
\usepackage{amssymb,amsmath,amsthm}
\usepackage{graphicx,color}
\title{{\bf Simple Groups Whose Prime Graph or Solvable Graph is Split}}
\author{ Mark L. Lewis, J. Mirzajani, A. R. Moghaddamfar, A. V. Vasil'ev, M. A. Zvezdina }

\newtheorem{theorem}{Theorem}[section]

\newtheorem{corollary}[theorem]{Corollary}

\newtheorem{proposition}[theorem]{Proposition}

\newtheorem{lm}[theorem]{Lemma}

\newtheorem*{coj*}{Conjecture}
\begin{document}
\newcommand{\f}{\frac}
\newcommand{\sta}{\stackrel}
\maketitle
\begin{abstract}
\noindent A graph is split if there is a partition of its vertex set into a clique and an independent set.
The present paper is devoted to the splitness of some graphs related to finite simple groups, namely, prime graphs and solvable graphs, and their compact forms.
It is proved that the compact form of the prime graph of any finite simple group is split.
\end{abstract}

\def\thefootnote{ \ }
\footnotetext{{\em $2000$ Mathematics Subject Classification}:
20D05, 05C25.}
\section{\sc Introduction}
In this paper we consider only simple finite graphs.
A graph $\Gamma$ is a {\em split graph} if and only if its vertex set can be partitioned into a complete and an independent set (either set can be empty). We refer to such a partition as a {\em split partition}.
Split graphs were introduced independently by F${\rm \ddot{o}}$ldes and Hammer~\cite{Foldes-Hammer}, and by Tyshkevich and Chernyak~\cite{79TysChe}. Split graphs are a popular subclass of perfect graphs,
and they can be characterized by a combination of small forbidden subgraphs (see Proposition \ref{forbidden}). Hammer and Simeone \cite{81HamSim} characterized split graphs in terms of their degree sequences.
In fact, split graphs are those graphs for which equality holds in the $m$-th Erd${\rm \ddot{o}}$s-Gallai inequality, where $m=m(\Gamma)=\operatorname{max}\{i:d_i\geqslant i-1\}$ for the degree sequence $\{d_i\}$ of a graph $\Gamma$. It is also proved in \cite{81HamSim} that, given a split graph $\Gamma$, the number $m(\Gamma)$ is equal to both chromatic number $\chi(\Gamma)$ and the number of vertices in the maximal independent set of $\Gamma$.
The degree sequence characterization implies a linear-time recognition for split graphs (there is an algorithm for determining whether a graph on $n$ vertices is a split graph, which has time complexity $O(n)$ if the degree sequence is given).

The present paper concerns with the graphs associated with finite groups. Let $\pi(G)$ denote the set of all prime divisors of the order of a finite group $G$. The {\em prime graph}
(or the {\em Gruenberg-Kegel} graph) ${\rm GK}(G)$ consists of the vertex set $\pi(G)$ and the set of edges $\{r, s\}$ such that there is an element of order $rs$ in $G$. It is clear that two primes $r, s\in\pi(G)$ are adjacent in ${\rm GK}(G)$ if and only if $G$ has a cyclic subgroup of order divisible by $rs$. Substituting `{\em cyclic}'
by `{\em solvable}' in this definition, we come to the notion of {\em solvable graph}. Thus, the solvable graph $\mathcal{S}(G)$ of a finite group $G$ has the vertex set $\pi(G)$, and two primes $r$ and $s$ are adjacent in $\mathcal{S}(G)$ if and only if $G$ has a solvable subgroup of order divisible by $rs$.
Note that ${\rm GK}(G)$ is a subgraph of $\mathcal{S}(G)$.

The notion of a prime graph was introduced in 1970s by Gruenberg and Kegel. Later on, Williams \cite{w} and Kondrat'ev \cite{k} obtained the classification of finite simple groups with disconnected prime graph. Abe and Iiyori introduced solvable graphs in \cite{abe-iiyori} and showed that the solvable graph of a finite group is always connected, and it is not complete when we restrict the base group to a nonabelian finite simple group. Both types of graphs yield rich information on the group structure
(see, e.g., \cite{lucido,05Vas, w}). Moreover, sometimes finite simple groups can be characterized by their prime graphs (see \cite{Hagie,Zavarnitsin,Zvezdina}). For these reasons, the prime and solvable graphs of finite (simple)
groups have been studied extensively for the last 40 years.
Gruber et al.~\cite{Gruber} characterized prime graphs of solvable groups.
Vasil'ev and Vdovin established an arithmetic criterion of adjacency and described all cocliques of maximal size in prime graph of every finite nonabelian simple group \cite{vasile-v,11VasVd.t}. Amberg and Kazarin \cite{13AmKaz} used this approach for the solvable graphs of finite simple groups and described the maximal cocliques in these graphs.

We investigate splitness of the prime graphs and solvable graphs of finite simple groups. 



\medskip
\noindent {\bf Theorem A.} \ {\em The prime graph ${\rm GK}(G)$ and the solvable graph ${\cal S}(G)$ of any alternating and symmetric group $G$ is split.}

\medskip
\noindent {\bf Theorem B.} \ {\em The prime graph ${\rm GK}(G)$ of any sporadic simple group $G$ is split. The solvable graph ${\cal S}(G)$ of a sporadic simple group $G$ is split, except for the following simple groups:  $M_{22}$, $M_{23}$, $M_{24}$, $Co_3$, $Co_2$,
$Fi_{23}$, $Fi_{24}'$, $B$, $M$ and $J_4$.}

\medskip
Now let $G$ be a finite simple group of Lie type over the field of order $q$, where $q$ is a power of a prime $p$. Then all numbers in $\pi(G)$, except for $p$, are primitive prime divisors for the numbers of the form $q^i-1$ for some $i$ (for definitions the reader is referred to Section \ref{pre-Lie}). Due to the properties of primitive prime divisors, it is easy to see that in most cases the prime graph ${\rm GK}(G)$ is not split (see Proposition \ref{nonspl-GK}).
However, the situation changes if we consider the {\em compact form} ${\rm GK_c}(G)$ of this graph.
Here by the {\em compact form} $\Gamma_{\rm c}$ of a graph $\Gamma$ we mean the quotient graph $\Gamma_{\rm c}=\Gamma/_\equiv$ with respect to the following equivalence relation on the vertex set $V_\Gamma$ of $\Gamma$: for every $u,v\in V_\Gamma$ we put $u\equiv v$ if $u^{\bot}=v^{\bot}$, where $a^{\bot}$ is the ball of radius 1 with center $a$ in $\Gamma$.

\medskip
\noindent {\bf Theorem C.} \ {\em The
graph ${\rm GK}_{\rm c}(G)$ of a finite simple group $G$ of Lie type is split.}

\medskip
Note that if $G$ is an abelian simple group, i.e. a group of prime order, then the graph ${\rm GK}(G)=\mathcal{S}(G)$ is a singleton, so it is split. It is also clear that a split graph has the split compact form. Thus, Theorems A, B and C yield the following:

\medskip
\noindent {\bf Theorem D.} \ {\em The
graph ${\rm GK_c}(G)$ of any finite simple group $G$ is split.}

\medskip

Although the compact form ${\cal S}_{\rm c}(G)$ of the solvable graph of a group $G$ is ``close'' to ${\rm GK_c}(G)$,
our study shows that there are examples of nonsplitness of ${\cal S}_{\rm c}(G)$, where $G$ is a sporadic group (see Proposition \ref{Sc_sporadic}) or a group of Lie type (see Section \ref{Sc_nonsplit}) 
Moreover, due to the positive solution of the Artin conjecture on primitive roots for almost all primes \cite{heath-brown}, one can construct infinitely many such examples (see Proposition \ref{nonspl-Sc}).


\section{\sc Preliminaries}
Only basic concepts about graphs and groups will be needed for
this paper. They can be found in any textbook about the Graph
Theory or Group Theory, for instance see \cite{Gor, west}.

Let $\Gamma=(V_\Gamma, E_\Gamma)$ be a graph with vertex set $V_\Gamma$ and edge set $E_\Gamma$. Given a set of vertices $X\subseteq V_\Gamma$, the subgraph induced by $X$ is written $\Gamma [X]$. We often identify a subset of vertices with the subgraph induced by that subset, and vice versa. A subset $X$ of $V_\Gamma$ is called a {\em clique} if the induced subgraph $\Gamma[X]$ is complete, and it is {\em independent}  if $\Gamma[X]$ is a null graph (i.e. a graph without edges). As usual, we denote by $K_n$ the complete graph on $n$ vertices.

A graph $\Gamma$ is a split graph if and only if there is a partition $V_\Gamma=C\uplus I$, where $C$ is a complete and $I$ an independent set (either of which might be empty). Thus $\Gamma$ can be `split' into a complete and an independent set. Any partition of the vertex set of a split graph into a complete and an independent set is called a {\em split partition}. In a particular case, a split partition $V_\Gamma=C\uplus I$ is {\em special} if every vertex in $I$ is not adjacent to at least one vertex in $C$. Note that, every split graph has a special split partition, because if there is a vertex in $I$ adjacent to any element in $C$, it can be moved to $C$. It should be pointed out that, a split partition (or a special split partition) of a split graph is not unique, but it is always possible to choose a partition such that $C$ is a clique of maximum size.

The following result is an immediate consequence of the definition of compact form given in Introduction.
\begin{proposition}\label{p-2.1}
 If a graph $\Gamma$ is split, then its compact form $\Gamma_{\rm c}$  is also split.
Conversely, if the compact form  $\Gamma_{\rm c}$ is split with split partition $C\uplus I$, then $\Gamma$ is split
if every vertex in $I$ is a singleton.
\end{proposition}

The set of all element orders of a finite group $G$ is called the {\em spectrum} of $G$ and is denoted by $\omega(G)$. It is clear that the set $\omega(G)$ is {\em closed} and {\em partially ordered} by divisibility relation, hence, it is uniquely determined by $\mu(G)$, the subset of its {\em maximal} elements. For each natural number $n$ we denote by $\pi(n)$ the set of all prime divisors of $n$, and put $\pi(G)=\pi(|G|)$ which is called the {\em prime spectrum} of $G$. The spectrum of $G$ determines the prime graph ${\rm GK}(G)$ of $G$, whose vertex set is the prime spectrum of $G$, and two distinct vertices $p$ and $q$ are joined by an edge (briefly, adjacent) if and only if $pq\in \omega(G)$. For primes $r, s\in \pi(G)$, we will write $r\sim s$ if $r$ and $s$ are adjacent in ${\rm GK}(G)$. We denote by $s(G)$ the number of connected components of ${\rm GK}(G)$ and by $\pi_i=\pi_i(G)$, $i= 1, 2, \ldots, s(G)$, the set of vertices of $i$th connected component. We often identify a connected component of a prime graph with its vertex set, and vice versa. If $2\in \pi(G)$, then we assume that $2\in \pi_1(G)$.  We denote by $\mu_i= \mu_i(G)$ the set of all $n\in \mu(G)$ such that  $\pi(n)\subseteq \pi_i$. The vertex set of connected components of prime graphs associated with finite simple groups are listed in \cite{k} and \cite{w} (see the improved list in \cite{newmaz}).

The solvable graph $\mathcal{S}(G)$ of a group $G$ is a graph with vertex set $\pi(G)$ in which two distinct primes $r$ and $s$ are adjacent if
and only if $G$ has a solvable subgroup of order divisible by $pq$. Adjacency of two vertices $r$ and $s$ in a solvable graph $\mathcal{S}(G)$ is written by $r\approx s$.

Denote the  $i$th connected components of
${\rm GK}(G)$ by $${\rm GK}_i(G)=(\pi_i(G), E_i(G)),  \ \ i=1, 2,
\ldots, s(G).$$
In \cite{suz}, Suzuki studied the structure of prime graph
associated with a finite simple group, and proved the following
interesting result. It is worth stating that his proof of this
result does not use the classification of finite simple groups.

\begin{proposition}[Suzuki] {\rm (\cite[Theorem B]{suz})}\label{prop1}   Let $L$ be a
finite simple group whose prime graph ${\rm GK}(L)$ is
disconnected and let ${\rm GK}_i(G)$ be a connected component of
${\rm GK}(G)$ with $i\geqslant 2$. Then ${\rm GK}_i(G)$ is a
clique.
\end{proposition}

 A similar result as Proposition {\rm \ref{prop1}} can
be found in {\rm \cite[Lemma 4]{km}}.  Note that Proposition $\ref{prop1}$ is true for
all {\em finite groups} not only for finite simple groups (see \cite{w}).
According to Proposition $\ref{prop1}$,  the prime
graph of an arbitrary finite simple group $G$ has the following form:
$${\rm GK}(G)=\bigoplus_{i=1}^{s}{\rm GK}_i(G)= {\rm GK}_1(G)\oplus K_{n_2}\oplus \cdots\oplus K_{n_s},$$
where $n_i=|\pi_i(G)|$ and $s=s(G)$.

In general case, if ${\rm
GK}(G)$ is a complete graph, then it may be viewed as a split
graph. For example, if $G$ is a nilpotent group or
$G=P_1\times P_2\times \cdots \times P_r$, where $P_i$'s are isomorphic nonabelian simple groups and  $r\geqslant 2$, then ${\rm GK}(G)$ is a complete
graph, and so is split.

In \cite{Foldes-Hammer}, the authors provided the following finite
forbidden subgraph characterization of split graphs.
\begin{proposition} [Forbidden Subgraph Characterization] {\rm (see \cite{Foldes-Hammer})} \label{forbidden} \
A graph is a split graph if and only if it contains no induced
subgraph isomorphic to $2K_2$ (two parallel edges), $C_4$ (a
square), or $C_5$ (a pentagon).
\end{proposition}

It follows from the definition (or the forbidden subgraph
characterization) that the complement, and every induced subgraph
of a split graph is split. Moreover, by  Proposition \ref{forbidden},
we immediately have the following.
\begin{corollary}\label{disconnected-split}
Let $\Gamma$ be a
split graph with a split partition $V=C\uplus I$. If $\Gamma$ has
more than one connected component, then the connected components
consist of the following possibilities: a connected component
containing $C$ and $\{ v$, a single
prime$\}\subseteq I$.
\end{corollary}

\begin{corollary}\label{three-split}
A graph with at most three vertices is always split. In
particular, the prime (resp. solvable) graph ${\rm GK}(G)$ (resp. $\mathcal{S}(G)$) of a group $G$,
with $|\pi(G)|\leqslant 3$, is split.
\end{corollary}

It follows immediately from the definition that if $G$ is a
solvable group, then $\mathcal{S}(G)$ is a complete graph, in
particular, $\mathcal{S}(R(G))$ is a complete graph, where $R(G)$
denotes the {\em solvable radical} of $G$. However, the converse is not
necessarily true. For example, if $G$ is a simple group, then the
solvable graph $\mathcal{S}(G\times G)$ is complete, while
$G\times G$ is not a solvable group. Thus from now on, we shall
restrict our attention to the solvable graphs associated with
{\em nonsolvable} groups. We continue with some elementary facts concerning the adjacency for
two prime divisors in solvable graph of subgroups and factor groups of $G$. The
following lemma is taken from \cite{abe-iiyori}.

\begin{lm}{\rm (\cite[Lemma 2]{abe-iiyori})}\label{subgraph}
Let $G$ be a finite group. Let $H$ and $N$ be two subgroups of $G$
with $N\unlhd G$. Then the following statements hold:
\begin{itemize}
\item[$(1)$] If $p$ and $q$ are adjacent in $\mathcal{S}(H)$
for $p, q\in\pi(H)$, then $p$ and $q$ are adjacent in
$\mathcal{S}(G)$, in other words, $\mathcal{S}(H)$ is a subgraph
of $\mathcal{S}(G)$.
\item[$(2)$] If $p$ and $q$ are adjacent in $\mathcal{S}(G/N)$
for $p, q\in\pi(G/N)$, then $p$ and $q$ are adjacent in
$\mathcal{S}(G)$, in other words, $\mathcal{S}(G/N)$ is a subgraph
of $\mathcal{S}(G)$.
\item[$(3)$] For $p\in\pi(N)$
and $q\in\pi(G)\backslash \pi(N)$, $p$ and $q$ are adjacent in
$\mathcal{S}(G)$.
\end{itemize}
\end{lm}

\begin{corollary}\label{cor1} Let $G$ be a finite group and let $\Delta(G)$ be the set of vertices of the
solvable graph $\mathcal{S}(G)$ which are joined to all other
vertices. Then, we have:
\begin{itemize}
\item[{\rm (1)}] $\pi(R(G))\subseteq \Delta(G)$.
\item[{\rm (2)}] If $\mathcal{S}(G)$ has a special split partition $\pi(G)=C\uplus I$, then
$\Delta(G) \subseteq C$, in particular $\pi(R(G))\subseteq
C$.
\end{itemize}
\end{corollary}
\begin{proof} (1) This is immediate by applying Lemma \ref{subgraph} (3) to $R(G)$, and noting that $\mathcal{S}(R(G))$ is complete.

(2) The result follows by the definition and part (1).
\end{proof}


\section{\sc Alternating and Symmetric Groups}
The adjacency criterion for two prime divisors in the prime
graph of an alternating or symmetric group is obvious and can be stated as
follows (see \cite[Proposition 1.1]{vasile-v}).
\begin{proposition}\label{alt-symm} Let $n$ be a natural number and let $p$ and $q$ be two distinct odd primes
in $\pi(G)$, where $G$ is an alternating or symmetric group of degree $n$. Then, we have:
\begin{itemize}
\item[{\rm (1)}]  $p$ and $q$ are adjacent in ${\rm GK}(G)$  if and only if $p + q\leqslant n$;

\item[{\rm (2)}]  if $G$ is symmetric, then $2$ and $p$ are adjacent in ${\rm GK}(G)$ if and only if $2+p\leqslant
n$;

\item[{\rm (3)}] if $G$ is alternating, then $2$ and $p$ are adjacent in ${\rm GK}(G)$ if and only if $4+p\leqslant
n$.
\end{itemize}
\end{proposition}

For a natural number $n$, we denote by $l_n$ the largest prime not
exceeding $n$ and by $s_n$ denote the smallest prime greater than
$n$. We also denote by $\lfloor x\rfloor$ the integer part of
$x$, i.e., the greatest integer less than or equal to $x$.
Proposition \ref{alt-symm} leads to the following result.

\begin{proposition} \label{th-3-2} Let $G$ be a symmetric or alternating group of degree $n\geqslant 2$. Then the prime graph and the solvable graph of $G$ are split graphs with a
split partition $C\uplus I$, where  $$C=\{2, 3, 5, \ldots,
l_{\lfloor\frac{n}{2}\rfloor}\}\textit{ and }
I=\{s_{\lfloor\frac{n}{2}\rfloor},  \ldots, l_n\}.$$ \end{proposition}
\begin{proof}  In view of Proposition \ref{alt-symm},  it is clear that $C$ is a complete set and $I$ an
independent set in ${\rm
GK}(G)$.
Also, from Proposition \ref{alt-symm}, it
is easy to see that $C\subset \pi(G)$ is a
complete set in $\mathcal{S}(G)$, because ${\rm
GK}(G)$ is a subgraph of $\mathcal{S}(G)$. Hence it
will be enough to prove that $I$ is an independent set in $\mathcal{S}(G)$. Let $p, q\in I$ and
$p\approx q$ in $\mathcal{S}(G)$. Then, from the
definition, there exists a solvable subgroup $K$ of $G$ whose
order is divisible by $pq$. We now consider a $\{p, q\}$-Hall
subgroup $K_0$ of $K$, which has order $pq$. Since $p\nmid q-1$
and $q\nmid p-1$, clearly $K_0$ is a cyclic group and so $p\sim
q$ in ${\rm GK}(G)$. It now follows from Proposition
\ref{alt-symm} that $p+q\leqslant n$, and this is a contradiction.
\end{proof}

Now Theorem A follows immediately from Proposition \ref{th-3-2}.


\section{\sc Sporadic Simple Groups}
Information on the adjacency of vertices in the prime graph and solvable graph of
a sporadic simple group can be found in
\cite{atlas}.

\begin{proposition}\label{sporadic} The prime graph of any sporadic simple group is a split graph.
\end{proposition}
\begin{proof} Using information from \cite{atlas}, we have determined a special
split partition for the prime graph of every sporadic simple
groups as in Table 1.
\end{proof}

\begin{center}
{
{\bf Table 1.} {\em A special split partition of prime graph of a sporadic simple group. }\\[0.5cm]
\begin{tabular}{|ll|l|l|}
\hline
Name & Symbol & $C$ & $I$\\
\hline
Mathieu groups& $M_{11}$ & $\{2\} $ & $\{3, 5, 11\}$ \\
&$M_{12}$ &  $\{2\}$  & $\{3, 5, 11\}$\\
&$M_{22}$ & $\{2\}$ &  $\{3, 5, 7, 11\}$\\
&$M_{23}$ & $\{2, 3\}$  &$\{5, 7, 11, 23\}$ \\
&$M_{24}$ & $\{2, 3\}$ & $\{5, 7, 11, 23\}$\\
Janko groups&$J_1$ & $\{2, 3\}$ & $\{5, 7, 11, 19\}$\\
&$J_2$ & $\{2, 3\}$ & $\{5, 7\}$\\
&$J_3$ & $\{2, 3\}$ & $\{5, 7, 19\}$\\
&$J_4$ & $\{2, 3, 5\}$& $\{7, 11, 23, 29, 31, 37, 43\}$ \\
Higman-Sims group&$HS$ & $\{2, 3\}$  & $\{5, 7, 11\}$\\
McLaughlin group&$M^cL$ & $\{2, 3\}$  & $\{5, 7, 11\}$\\
Suzuki group&$Suz$ & $\{2, 3\}$ & $\{5, 7, 11, 13\}$\\
Rudvalis group&$Ru$ & $\{2, 3\}$ & $\{5, 7, 13, 29\}$\\
Held group&$He$  & $\{2, 3\}$& $\{5, 7, 17\}$\\
Lyons group&$Ly$ & $\{2, 3\}$ & $\{5, 7, 11, 31, 37, 67\}$ \\
O'Nan group&$O'N$ &$\{2, 3\}$ & $\{5, 7, 11, 19, 31\}$ \\
Conway groups&$Co_1$ &$\{2, 3, 5\}$ & $\{7, 11, 13, 23\}$\\
&$Co_2$ & $\{2, 3\}$& $\{5, 7, 11, 23\}$\\
&$Co_3$  &$\{2, 3\}$ & $\{5, 7, 11, 23\}$ \\
Fischer groups&$Fi_{22}$ &$\{2, 3\}$ & $\{5, 7, 11, 13\}$\\
&$Fi_{23}$ & $\{2, 3, 5\}$& $\{7, 11, 13, 17, 23\}$ \\
&$Fi'_{24}$ & $\{2, 3, 5\}$ & $\{7, 11, 13, 17, 23, 29\}$ \\
Harada group&$F_5$ & $\{2, 3, 5\}$ & $\{7, 11, 19\}$\\
Thompson group &$F_3$ & $\{2, 3\}$ & $\{5, 7, 13, 19, 31\}$\\
Babymonster group&$F_2$ &$\{2, 3, 5\}$  & $\{7, 11, 13, 17, 19, 23, 31, 47\}$ \\
Monster group&$F_1$ & $\{2, 3, 5, 7\}$ & $\{11, 13, 17, 19, 23, 29, 31, 41, 47, 59, 71\}$\\
\hline
\end{tabular}}
\end{center}

\begin{proposition}\label{sporadic-solvable}
The solvable graph of any sporadic simple group is a split
graph, except for the groups  $M_{22}$, $M_{23}$, $M_{24}$, $Co_3$, $Co_2$,
$Fi_{23}$, $Fi_{24}'=F_{3+}$, $B=F_{2+}$, $M=F_1$ and $J_4$.
\end{proposition}
\begin{proof} Using information from \cite{atlas}, we have determined a special
split partition for the solvable graph of any sporadic simple
group, except for the groups  $M_{22}$, $M_{23}$, $M_{24}$, $Co_3$, $Co_2$,
$Fi_{23}$, $Fi_{24}'=F_{3+}$, $B=F_{2+}$, $M=F_1$ and $J_4$, as
in Table 2. On the other hand, the solvable graph
$\mathcal{S}(M_{22})$ is depicted in Figure 1 (a). Recall that $7:3\leqslant L_2(7)\leqslant A_7\leqslant M_{22}$,
$11:5\leqslant L_2(11)\leqslant M_{22}$, $S_3\leqslant
A_5\leqslant A_7\leqslant M_{22}$, $D_{10}\leqslant A_5\leqslant
A_7\leqslant M_{22}$, and $2^3:7\leqslant 2^3:L_3(2)\leqslant
M_{22}$. Then the subgraph induced by  $\{3, 5, 7, 11\}$ is $2K_2$,
which is a forbidden subgraph for a split graph.

\vspace{0.95cm} {  \setlength{\unitlength}{4.5mm}
\begin{picture}(0,0)(-4,0)
\linethickness{0.3pt} %
\put(0,0){\circle*{0.4}}
\put(7,1.5){\circle*{0.4}}
\put(7,-1.5){\circle*{0.4}}
\put(2.4,0){\circle*{0.4}}
\put(4.8,0){\circle*{0.4}}

\put(-0.2,0.4){$11$}%
\put(7.3,1.7){$3$}%
\put(7.3,-2.2){$7$}%
\put(2.3,0.4){$5$}%
\put(4.5,0.4){$2$}%

\put(7,1.5){\line(-3,-2){2.27}}
\put(7,-1.5){\line(-3,2){2.27}}
\put(0,0){\line(1,0){4.8}}
\put(7,1.5){\line(0,-1){3}}
\put(15,0){\circle*{0.4}}
\put(22.2,0){\circle*{0.4}}
\put(17.4,0){\circle*{0.4}}
\put(19.8,0){\circle*{0.4}}
\put(14,0.6){$\{11\}$}%
\put(22.7,0){$\{3, 7\}$}%
\put(16.7,0.6){$\{5\}$}%
\put(19.1,0.6){$\{2\}$}%

\put(15,0){\line(1,0){7.2}}
\put(15,-2.5){ (b)  \  $\mathcal{S}_{\rm c}(M_{22})$}

\put(0,-2.5){ (a)  \  $\mathcal{S}(M_{22})$}
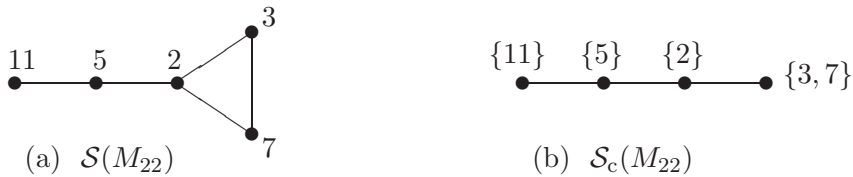
\put(1.5,-4){ {\bf Fig. 1}  \  The solvable graph and its compact form of $M_{22}$.}
\end{picture}}\vspace{2.5cm}

Similarly, for the
other sporadic groups, we have also determined a subset $W$ of $\pi(G)$ below,
so that the subgraph induced by $W$ is $2K_2$:
\begin{center}
\begin{tabular}{ll|l}
\hline
Name & Symbol(s) & $W$\\
\hline
Mathieu & $M_{22}$  & $\{3, 5, 7, 11\}$  \\
Mathieu & $M_{23}$ &   $\{2, 7, 11, 23\}$\\
Mathieu & $M_{24}$  &  $\{3, 7, 11, 23\}$ \\
Conway& $Co_{3}$  &  $\{3, 7, 11, 23\}$ \\
Conway& $Co_{2}$  &  $\{3, 7, 11, 23\}$ \\
Fischer& $Fi_{23}$ &   $\{3, 7, 11, 23\}$ \\
Fischer& $Fi_{24}'=F_{3+}$  &   $\{7, 11, 23, 29\}$ \\
Fischer/Sims, Leon & $B=F_{2+}$ &  $\{7, 11, 23, 29\}$ \\
Fischer, Griess &$M=F_1$ &   $\{7, 29, 59, 71\}$  \\
Janko/Norton, Parker, Benson, Conway, Thackray &$J_4$ & $\{7, 11, 23, 43\}$ \\
\hline
\end{tabular}\end{center}
The proof is complete. \end{proof}

\begin{center}
{\bf Table 2.} {\em A special split partition of solvable graph of a sporadic simple group. }\\[0.5cm]
\begin{tabular}{|ll|l|l|}
\hline
Name & Symbol(s) & $C$ & $I$\\
\hline
Mathieu & $M_{11}$ & $\{2, 5\} $ & $\{3, 11\}$  \\
Mathieu&$M_{12}$ &  $\{2, 5\}$  & $\{3, 11\}$   \\
Hall, Janko & $J_2=HJ=F_{5-}$ & $\{2, 3\}$ & $\{5, 7\}$\\
Suzuki &$Suz$ & $\{2, 3, 5\}$ & $\{7, 11, 13\}$ \\
Higman, Sims&$HS$ & $\{2, 3, 5\}$  & $\{ 7, 11\}$ \\
McLaughlin&$M^cL$ & $\{2, 3, 5\}$  & $\{7, 11\}$ \\
Conway, Leech& $Co_1=F_{2-}$ &$\{2, 3, 5, 11\}$ & $\{7, 13, 23\}$ \\
Held/Higmann, McKay&$He$  & $\{2, 3\}$& $\{5, 7, 17\}$\\
Fischer&$Fi_{22}$ &$\{2, 3, 5\}$ & $\{7, 11, 13\}$\\
Harada, Norton/Smith&$HN=F_{5+}$ & $\{2, 3, 5, 7\}$ & $\{11, 19\}$\\
Thompson/Smith  &$Th=F_{3|3}$ & $\{2, 3, 5\}$ & $\{7, 13, 19, 31\}$ \\
Janko& $J_1$ & $\{2, 3, 5\}$ & $\{7, 11, 19\}$\\
O'Nan/Sims &$O'N$ &$\{2, 3, 5\}$ & $\{7, 11, 19, 31\}$ \\
Janko/Higmann, McKay&$J_3$ & $\{2, 3, 5\}$ & $\{7, 19\}$\\
Lyons/Sims &$Ly$ & $\{2, 3, 11\}$ & $\{5, 7, 31, 37, 67\}$ \\
Rudvalis/Conway, Wales &$Ru$ & $\{2, 3, 7\}$ & $\{5, 13, 29\}$ \\
\hline
\end{tabular}
\end{center}
\vspace{0.4cm}

Propositions \ref{sporadic} and \ref{sporadic-solvable} imply the statement of Theorem B.
Proposition \ref{sporadic} also yields that the compact form ${\rm GK_c}(L)$ of the prime graph of any sporadic group is split. The following proposition is concerned with the compact form ${\cal S}_{\rm c}(G)$ of the solvable graph of a sporadic group.
\begin{proposition}\label{Sc_sporadic}
The graph ${\cal S}_{\rm c}(G)$ of any sporadic simple group, except $M_{23}$, $M_{24}$, $Co_3$, $Co_2$,
$Fi_{23}$, $Fi_{24}'=F_{3+}$, $B=F_{2+}$, $M=F_1$ and $J_4$, is split.
\end{proposition}
\begin{proof}
By Proposition \ref{p-2.1}, the compact form of a split graph is split. Therefore, the graph ${\cal S}_{\rm c}(G)$ is split for the groups $G$ listed in Table 2. The compact form ${\cal S}_{\rm c}(M_{22})$ is a path of length 3  which is split (Figure 1 (b)). On the other hand, if  $G$ is isomorphic to one of the groups:  $M_{23}$, $M_{24}$, $Co_3$, $Co_2$, $Fi_{23}$, $Fi_{24}'=F_{3+}$, $B=F_{2+}$, $M=F_1$ or $J_4$, then  ${\cal S}_{\rm c}(G)={\cal S}(G)$, which is not split.
\end{proof}

In the sequel, we need only consider the simple groups of Lie
type.


\section{\sc Preliminary Results on the Groups of Lie Type}\label{pre-Lie}
The greatest common divisor of natural numbers $m$ and $n$ is denoted by $(m, n)$.
If $n$ is a nonzero integer and $r$ is an odd prime with $(r, n) = 1$, then $e(r, n)$ denotes the multiplicative order of $n$ modulo $r$,  i.e., a minimal natural number $k$ with $n^k\equiv 1\pmod{r}$.
Given an odd integer $n$, we put $e(2, n) = 1$ if $n\equiv 1\pmod{4}$ and put $e(2,n) = 2$ if $n\equiv 3\pmod{4}$.
Fix an integer $n$ with $|n|>1$. A prime $r$ with $e(r, n)=i$ is called a {\em primitive prime divisor of} $n^i-1$.  We write $r_i(n)$  to denote some primitive prime divisor of $n^i-1$, if such a prime exists, and $R_i(n)$ to denote the set of all such divisors. Instead of $r_i(n)$ and $R_i(n)$ we simply write $r_i$ and $R_i$ if it does not lead to confusion.  Bang  \cite{Bang} and Zsigmondy \cite{Zsigmondy} proved that primitive prime divisors exist except for a few cases$^1$\footnote{$^1$In fact, Bang \cite{Bang} proved in 1886 that $n^i-1$ has a primitive prime divisor for all $n\geqslant 2$ and $i>2$ except for $n=2$ and $i=6$.
Then, Zsigmondy \cite{Zsigmondy} proved in 1892 that for coprime integers $a> b\geqslant 1$  and $i> 2$, there exists a prime $r$ dividing $a^i-b^i$ but not $a^k-b^k$ for $1\leqslant k<i$, except when $a=2$, $b=1$, and $i=6$.}\!\!\!\!.
\begin{theorem}\label{zsig}  {\rm  (Bang --Zsigmondy)}.
Let $n$ and $i$ be integers satisfying $|n|>1$ and $i\geqslant 1$. Then $R_i(n)\neq \emptyset$, except when
$(n, i)\in \{(2, 1), (2, 6), (-2,2), (-2,3), (3, 1), (-3,2)\}$.
\end{theorem}

Theorem \ref{zsig} has many applications; for instance,  see \cite{Artin} for applications of primitive prime divisors in finite group theory. In what follows, we will concentrate on the case when $L={^dL}_l(q)$ is a simple
group of Lie type of rank $l$ over the field with $q$ elements.  Our notation for these groups is borrowed from \cite{atlas}. In view of \cite[Thms. 9.4.10, 14.3.1]{Carter},
 the order of any finite simple group of Lie type $L$ of rank $l$ over the field ${\rm GF}(q)$ of characteristic $p$ is equal to
$|L|=q^N(q^{n_1}\pm 1) (q^{n_2}\pm 1) \cdots (q^{n_l}\pm 1)/d$ (see Table 3).
Therefore any prime divisor $r$ of $|L|$ distinct from the characteristic $p$ is a primitive divisor for $(q, i)$, for some natural number $i$. Thus Lemma  \ref{zsig} allows us to find prime divisors of $|L|$. Moreover, if $L$ is neither a Suzuki group nor a Ree group, Lemmas 1.2 and 1.3
in \cite{vasile-v} imply that for a fixed $i$, every two primitive prime divisors for $(q, i)$ are adjacent in ${\rm GK}(L)$.

Given a set of primes $\pi$ and an integer $a$, denote by $(a)_{\pi}$ the \emph{$\pi$-part} of $a$, i.e., the greatest divisor $m$ of $a$ such that $\pi(m)\subseteq\pi$.
\begin{lm}\label{divisors}
Let $q=p^a$, where $p$ is a prime and $a\geqslant 2$. Let $k>1$ be a natural number, such that $\pi(a)\nsubseteq\pi(k)$. Denote $a'=(a)_{\pi(k)}$ and assume that the sets $R_{ka}(p)$ and $R_{ka'}(p)$ are nonempty. Then $|R_k(q)|>1$.
\end{lm}

\begin{proof}
It is clear that $R_{ka}(p)\subseteq R_k(p^a)$. It suffices to show that $R_{ka}(p)\cap R_{ka'}(p)=\emptyset$ and $R_{ka'}(p)\subseteq R_k(p^a)$.
Let $\pi(a)\cap\pi(k)=\{r_1,\ldots,r_s\}$ (this set can be empty). Then $$a=r\cdot \prod\limits_{i=1}^{s} r^{m_i}_i \ \ \ \mbox{and} \ \ \ k=t\cdot \prod\limits_{i=1}^{s} r^{n_i}_i,$$ where $(r, t)=(r, r_i)=(t, r_i)=1$ for all $i=1, \ldots, s$.
Since $\pi(a)\nsubseteq\pi(k)$, we have that $r>1$, and so $a\neq a'$. Now $R_{ka}(p)\cap R_{ka'}(p)=\emptyset$ by definition.
Let $x$ be a primitive prime divisor of $p^{ka'}-1$,
where $$ka'=t\cdot\prod\limits_{i=1}^{s} r^{m_i+n_i}_i.$$ Then $x$ is a primitive prime divisor of $(p^a)^k-1=(p^a)^{t\cdot r^{n_1}_1\cdots r^{n_s}_s}-1$, since $(t, a)=1$. Therefore, $R_{ka'}(p)\subseteq R_k(p^a)$, and so $|R_k(q)|>1$.
\end{proof}

\begin{center}
 {\bf Table 3.} {\em The Orders of Finite Simple Groups of Lie
Type}\\[-0.3cm]
\[\begin{array}{|l|l|l|l|} \hline
\mbox{Group}& \mbox{Conditions} & \mbox{Other \ names} & \mbox{Order}\\
\hline & & & \\[-0.3cm] A_n(q) & n\geqslant 1 & {\rm {\rm PSL}}_{n+1}(q)=L_{n+1}(q)
 & \frac{1}{(n+1,q-1)}q^{n+1\choose 2}\prod\limits_{i=2}^{n+1}\left(q^i-1\right)
\\
& & =L^+_{n+1}(q)=A_n^+(q) & \\[0.3cm] B_n(q) & n\geqslant 2 &
P\Omega_{2n+1}(q)=\Omega_{2n+1}(q) &
\frac{1}{(2,q-1)}q^{n^2}\prod\limits^n_{i=1}\left(q^{2i}-1\right)\\[0.4cm]
C_n(q) & n\geqslant 2 & PSp_{2n}(q) &
\frac{1}{(2,q-1)}q^{n^2}\prod\limits^n_{i=1}\left(q^{2i}-1\right)\\[0.4cm]
D_n(q) & n\geqslant 3 & P\Omega_{2n}^+(q)=D_{n}^+(q) &
\frac{1}{(4,q^n-1)}q^{n(n-1)}(q^n-1)\prod\limits^{n-1}_{i=1}\left(q^{2i}-1\right)\\[0.2cm]
G_2(q) & & & q^6(q^6-1)(q^2-1)\\[0.2cm]
F_4(q) & & & q^{24}(q^{12}-1)(q^8-1)(q^6-1)(q^2-1)\\[0.2cm]
E_6(q) & & E_6^+(q) & \frac{1}{(3,q-1)}
q^{12}(q^9-1)(q^5-1)\times |F_4(q)|\\[0.2cm]
E_7(q) & & & \frac{1}{(2,q-1)}q^{39}(q^{18}-1)(q^{14}-1)(q^{10}-1)\\[0.2cm]
& & & \times |F_4(q)| \\[0.2cm]
E_8(q) & & & q^{96}(q^{30}-1)(q^{12}+1)(q^{20}-1)
(q^{18}-1)\\[0.2cm]
& & & (q^{14}-1)(q^{6}+1)\times |F_4(q)|\\[0.2cm] ^2A_n(q) & n\geqslant 2 &
{\rm PSU}_{n+1}(q)=U_{n+1}(q) & \frac{1}{(n+1,q+1)}q^{n+1\choose 2}\prod\limits^{n+1}_{i=2}\left(q^i-(-1)^i\right)\\[0.2cm]
& & =L^-_{n+1}(q)=A_n^-(q) & \\[0.2cm]
^2B_2(q) & q=2^{2m+1} & {\rm Sz}(q)={\sp2B}_2(\sqrt{q})& q^{2}(q^2+1)(q-1)\\[0.2cm]
^2D_n(q) & n\geqslant 2 & P\Omega_{2n}^-(q)=D_{n}^-(q) &
\frac{1}{(4,q^n+1)}q^{n(n-1)}(q^n+1)\prod\limits_{i=1}^{n-1}\left(q^{2i}-1\right)\\[0.2cm]
^3D_4(q) & & & q^{12}(q^8+q^4+1)(q^6-1)(q^2-1)\\[0.2cm]
^2G_2(q) & q=3^{2m+1} & R(q)={\sp2G}_2(\sqrt{q})& q^{3}(q^3+1)(q-1)\\[0.2cm]
^2F_4(q) & q=2^{2m+1} & ^2F_4(\sqrt{q}) & q^{12}(q^6+1)(q^4-1)(q^3+1)(q-1)  \\[0.2cm]
^2E_6(q) & & E_6^-(q) &
\frac{1}{(3,q+1)}q^{12}(q^9+1)(q^5+1)\times |F_4(q)|\\[0.3cm]
\hline
\end{array}
\]
\end{center}

We define two functions $\nu$ and $\eta$ on $\mathbb{N}$ as
follows:
$$\nu(n)=\left\{ \begin{array}{ll} n & \mbox{if}  \ n\equiv 0\!\!\!\pmod{4},
\\[0.1cm]
\frac{n}{2} & \mbox{if} \ n\equiv 2\!\!\!\pmod{4},
\\[0.1cm]
2n& \mbox{if} \ n \equiv 1\!\!\!\pmod{2},\\ \end{array} \right.  \ \ \ \ \ {\rm and} \ \ \ \ \ \eta(n)=\left\{ \begin{array}{ll}
n & \mbox{if} \ n \ \mbox{is odd},
\\[0.1cm]
\frac{n}{2}& \mbox{if} \ n \ \mbox{is even.}\\ \end{array}
\right.
$$
Now, we put
$$
\nu_{\epsilon}(n)=\left\{ \begin{array}{ll}
n & \mbox{if} \ \epsilon=+,
\\[0.1cm]
\nu(n)& \mbox{if} \ \epsilon=-.\\ \end{array}
\right.
$$
Given a simple classical group $L$ over a field of order $q$ and a prime $r$ coprime to $q$, we put

$$\varphi (r, L)=\left\{ \begin{array}{ll}
e(r, \epsilon q) & \mbox{if} \ L=L_n^\epsilon (q),
\\[0.1cm]
\eta(e(r, q)) & \mbox{if} \ L \ \mbox{is symplectic or orthogonal,}\\ \end{array}
\right.
$$
and
$$\delta (L)=\left\{ \begin{array}{ll}
\pi(\epsilon q-1) & \mbox{if} \ L=L_n^\epsilon (q),
\\[0.1cm]
\pi((2, q-1)) & \mbox{if} \ L \ \mbox{is symplectic or orthogonal.}\\ \end{array}
\right.
$$
For a classical group $L$, we put ${\rm prk}(L)$ to denote its dimension if $L$ is a linear or unitary group, and its Lie rank if $L$ is a symplectic or orthogonal group.

The following two lemmas taken from \cite[Lemmas 2.2 and 2.4]{15Vas}, in fact, follow from the results of \cite{vasile-v, 11VasVd.t}.

\begin{lm}\label{vasi-algebra-1}
Let $L$ be a simple classical group over a field of order $q$ and characteristic $p$.
Suppose that $r_{i_0}\in R_i(q)$ and $r_{j_0}\in R_j(q)$ are distinct primes such that $r_{i_0}, r_{j_0}\notin \delta(L)$. Then, we have:
\begin{itemize}
\item[{\rm (i)}]  If $r_{i_0}\sim r_{j_0}$ in ${\rm GK}(L)$,  then for all distinct odd primes $r_i\in R_i(q)$ and  $r_j\in R_j(q)$, $r_i\sim r_j$ in ${\rm GK}(L)$.

\item[{\rm (ii)}] If $r_{i_0}\sim p$ in ${\rm GK}(L)$, then
for all odd primes $r_i\in R_i(q)$, $r_i\sim p$ in ${\rm GK}(L)$.
\end{itemize}
\end{lm}

\begin{lm}\label{vasi-algebra}
Let $L$ be a simple classical group over a field of order $q$ and characteristic $p$, and let ${\rm prk}(L) = n\geqslant   4$.
\begin{itemize}
\item[{\rm (i)}]  If $r\in \pi(L)\setminus \{p\}$,  then  $\varphi (r, L)\leqslant n$.
\item[{\rm (ii)}] If $r$ and $s$ are distinct primes in $\pi(L)\setminus \{p\}$ with $\varphi (r, L)\leqslant   n/2$ and $\varphi (s, L)\leqslant   n/2$,
then $r$ and $s$ are adjacent in ${\rm GK}(L)$.
\item[{\rm (iii)}] If $r$ and $s$ are distinct primes in $\pi(L)\setminus \{p\}$ with $n/2<\varphi (r, L)\leqslant   n$ and $n/2<\varphi (s, L)\leqslant   n$,
then $r$ and $s$ are adjacent in ${\rm GK}(L)$ if and only if $e(r, q)=e(s, q)$.
\item[{\rm (iv)}] If $r$ and $s$ are distinct primes in $\pi(L)\setminus \{p\}$ and $e(r, q)=e(s, q)$, then
$r$ and $s$ are adjacent in ${\rm GK}(L)$.
\end{itemize}
\end{lm}

Another lemma gives necessary information on adjacency with characteristic in the prime graph of a classical group of Lie type.
\begin{lm}\label{lm_char}
Let $L$ be a simple classical group over a field of characteristic $p$, and let ${\rm prk}(L) = n\geqslant 4$.
Then $\varphi(r,L)>n/2$ for every prime $r$ nonadjacent to $p$ in ${\rm GK}(L)$.
\end{lm}
\begin{proof}
It follows from \cite[Proposition 6.3, Table 4]{vasile-v} (see also \cite[Lemma 2.6]{15Vas}).
\end{proof}

The properties of the solvable graphs of finite simple groups are studied in \cite{13AmKaz}. We will need the following assertions about finite simple linear groups, which are taken from~\cite[Lemmas 2.6 and 3.1]{13AmKaz}.

\begin{lm}\label{AK26}
Let $L=A_{n-1}(q)={\rm PSL}_n(q)$, and let $r=r_m(q)$ for $(m,q)\neq (6,2)$, $m>\lfloor\frac{n}{2}\rfloor$. Then the following hold:
\begin{itemize}
\item[{\rm (1)}] The Sylow $r$-subgroups of $L$ are cyclic;
\item[{\rm (2)}] If $A$ is a subgroup of order $r$ in $L$, then the index $|N_L(A):C_L(A)|$ divides $m$.
\end{itemize}
\end{lm}

\begin{lm}\label{AK31}
Let $H$ be a solvable subgroup of the group $L=A_{n-1}(q)={\rm PSL}_n(q)$, whose order is coprime to $q$. Let $r=r_m(q)$ and $s=r_l(q)$, where $m>l>\frac{n}{2}$. If the order of $H$ is divisible by $rs$, then $s$ divides $m$.
\end{lm}

The next lemma follows from \cite[Theorem 21.6]{supr}.
\begin{lm}\label{tor}
Let $G={\rm GL}_n(p)$, where $p$ and $n>3$ are primes. Let $H$ be a solvable $\{r,n\}$-subgroup of $G$, where $n=r_{n-1}(p)$ and $r\neq n$ does not divide $n\pm 1$ and $p^n-1$. Then $H$ is reducible.
\end{lm}


\section{\sc Proof of Theorem C}
In this section we prove that the compact form ${\rm GK_c}(L)$ of the prime graph of a simple group of Lie type $L$ is split.
Technical tools for determining ${\rm GK_c}(L)$ are Lemmas \ref{vasi-algebra-1}--\ref{lm_char} in the case when $L$ is a classical group with ${\rm prk}(L)\geqslant 4$, and \cite{vasile-v,11VasVd.t} otherwise.

Let $L$ be a group of Lie type over a field of order $q$ and characteristic $p$. Assume first that $L$ is a classical group with $n={\rm prk}(L)\geqslant4$.
Define $$J(L)=\{e(r,q) \mid n/2<\varphi(L)\leq n \} . $$
Lemma \ref{vasi-algebra} yields that the set $I=\{R_j(q) \mid j\in J(L)\}$ is a coclique in~${\rm GK_c}(L)$.
On the other hand, the set $\{p\}\cup\{r \mid \varphi(r,L)\leq n/2\}$ is a clique in the prime graph ${\rm GK}(L)$ by Lemmas~\ref{vasi-algebra-1}--\ref{lm_char}. Thus, the graph ${\rm GK_c}(L)$ is split for a classical group $L$ with $n={\rm prk}(L)\geqslant4$.

For all other groups of Lie type it can be deduced as follows.

\noindent {\em Type  $A_l^\epsilon$}  \ \ First, we assume that  $L=A_1(q)$, $q=p^n$.
We already know that
$$\mu(L) = \left\{p, \  (q-1)/(2, q-1), \ (q+1)/(2, q-1)\right\}.$$
The diagram of the compact form of ${\rm GK}(L)$ shown in Figure 2.
Thus
 ${\rm GK_c}(L)$ is always split, with a special split partition $C\uplus I$, where
 $$\begin{array}{lll}
C=\left\{\{p\}\right\}  &   \mbox{and}  &  I=\left\{\pi((q-1)/(2, q-1)), \ \pi((q+1)/(2, q-1))\right\},\\[0.2cm]
C=\left\{\pi((q-1)/(2, q-1))\right\}  & \mbox{and} &    I=\left\{\{p\}, \ \pi((q+1)/(2, q-1))\right\},\\[0.2cm]
C=\left\{\pi((q+1)/(2, q-1))\right\}  & \mbox{and} &   I=\left\{\{p\}, \ \pi((q-1)/(2, q-1))\right\}.\\[0.2cm]
\end{array}$$
Next, we assume that  $L=A_2^\epsilon (q)$,  $q=p^n$, and put $U_1=R_{\nu_{\epsilon} (1)}\setminus \{2, 3\}$, $U_2=R_{\nu_{\epsilon} (2)}\setminus \{2, 3\}$.  The corresponding diagram of the compact form for ${\rm GK}(L)$ is shown in Figure 3  (see \cite{11VasVd.t}).
Here and further, if an edge occurs under some condition, we draw such an edge as a dashed line and write an occurrence condition.
\begin{itemize}
\item[{\rm (a)}] If $(q-\epsilon 1)_3>3$,  then
$\{\{2\}, \{3\}, \{p\}, U_1\}$ is a clique in ${\rm GK_c}(A_2^\epsilon (q))$.  Clearly,
 ${\rm GK_c}(L)$ is a split graph, with a special split partition $C\uplus I$, where
$C=\{\{2\}, \{3\}, \{p\}, U_1\}$ and $I=\{U_2, R_{\nu_{\epsilon} (3)}\}$.

\item[{\rm (b)}] If $(q-\epsilon 1)_3=3$,  then the set $\{\{2\}, \{p\}, U_1\}$ is a clique in the
compact form for ${\rm GK}(L)$,  while $3$, $U_2$, and $R_{\nu_{\epsilon} (3)}$  are pairwise nonadjacent. Hence,  ${\rm GK_c}(L)$ is a split graph, with a special split partition $C\uplus I$, where
$C=\{\{2\}, \{p\}, U_1\}$ and $I=\{\{3\}, U_2, R_{\nu_{\epsilon} (3)}\}$.

\item[{(c)}] If $(q-\epsilon 1)_3=1$, i.e., either $(q+\epsilon 1)_3>1$ and $3\in R_{\nu_{\epsilon} (2)}\neq \{2\}$, or $p=3$.  As before, we
see that the set $\{\{2\}, \{p\}, U_1\}$ is a clique in the compact form for ${\rm GK}(L)$, while $U_2$ and $R_{\nu_{\epsilon} (3)}$ are nonadjacent.  Again,  ${\rm GK_c}(L)$ is a split graph, with a special split partition $C\uplus I$, where
$C=\{\{2\}, \{p\}, U_1\}$ and $I=\{\{3\}, U_2, R_{\nu_{\epsilon} (3)}\}$.
\end{itemize}

\vspace{1.5cm} {  \setlength{\unitlength}{4.5mm}
\begin{picture}(0,0)(-3,0)
\linethickness{0.3pt} %
\put(0.5,0){\circle*{0.4}}
\put(0.5,2){\circle*{0.4}}
\put(0.5,-2){\circle*{0.4}}
\put(20,-2){\circle*{0.4}}
\put(17,0){\circle*{0.4}}
\put(23,0){\circle*{0.4}}
\put(17,2){\circle*{0.4}}
\put(23,2){\circle*{0.4}}
\put(23,-2){\circle*{0.4}}
\put(1.3,1.7){ $\{p\}$}%
\put(1.3,-0.3){\small $\pi\left(\frac{q-1}{(2, q-1)}\right)$}%
\put(1.5,-2.3){\small $\pi\left(\frac{q+1}{(2, q-1)}\right)$}%
\put(19,-3.2){ $\{2\}$}%
\put(15.5,-0.3){\small $U_1$}%
\put(15.5,2){\small $U_2$}%
\put(23,2.3){ $\{3\}$}
\put(23,-0.5){ $\{p\}$}
\put(20,-2){\line(3,2){3}}
\put(20,-2){\line(3,4){3}}
\put(20,-2){\line(-3,2){3}}
\put(20,-2){\line(-3,4){3}}
\put(17,0){\line(1,0){6}}
\put(17,0){\line(3,1){6}}
\put(17,0){\line(0,1){2}}
\put(17,2){\line(1,0){0.3}}
\put(17.5,2){\line(1,0){0.3}}
\put(18,2){\line(1,0){0.3}}
\put(18.5,2){\line(1,0){0.3}}
\put(19,2){\line(1,0){0.3}}
\put(19.5,2){\line(1,0){0.3}}
\put(20,2){\line(1,0){0.3}}
\put(20.5,2){\line(1,0){0.3}}
\put(21,2){\line(1,0){0.3}}
\put(21.5,2){\line(1,0){0.3}}
\put(22,2){\line(1,0){0.3}}
\put(22.5,2){\line(1,0){0.3}}
\put(23,0){\line(0,1){0.3}}
\put(23,0.5){\line(0,1){0.3}}
\put(23,1){\line(0,1){0.3}}
\put(23,1.5){\line(0,1){0.3}}
\put(16.7,2.75){\small $(q-\epsilon 1)_3\neq 3\neq p$}
\put(23.5,1){\small $(q-\epsilon 1)_3>3$}
\put(23.5,-2.3){\small $R_{\nu_{\epsilon} (3)}$}
\put(-1.5,-5){ {\bf Fig. 2}  ${\rm GK_c}(L_2(q))$}
\put(15.5,-5){ {\bf Fig. 3}  ${\rm GK_c}(L_3^\epsilon (q))$.}
\end{picture}}\vspace{3cm}

\noindent {\em Type  $B_l$}  \ \  Assume first that $L=B_2(q)\cong C_2(q)$, with the base field of characteristic $p$ and order $q$.  Here, $|L|=\frac{1}{(2, q-1)}q^4(q^2-1)(q^4-1)$ and we have (see \cite[Lemmas  11 and 12]{newmaz} and \cite[Lemma 7]{mazsuchao}):
$$\mu(L)=\left\{\begin{array}{lll}
\{ (q^2+1)/(2, p-1), (q^2-1)/(2, p-1), p(q+1), p(q-1), p^2\} & \mbox{if} &  p=2, 3.\\[0.5cm]
\{ (q^2+1)/2, (q^2-1)/2, p(q+1), p(q-1)\} & \mbox{if} &  p\neq 2, 3.\\
\end{array}\right.$$
Thus ${\rm GK_c}(L)$ is split,  with split partition $C\uplus I$, where
$C=\{\{p\}, R_1, R_2\}$ and $I=\{R_4\}$.

Assume next that $L\in \{B_3(q), C_3(q)\}$.  In this case, $|L|=\frac{1}{(2, q-1)}q^9(q^2-1)(q^4-1)(q^6-1)$.
If $q=2$, then the group $L$ is not simple, and so we may assume that $q\geqslant  3$. If $q=3$,
then $\mu (L)=\{8, 12, 13, 14, 18, 20\}$ (see \cite[Lemma 2.1]{Staroletov}). This shows that ${\rm GK_c}(L)$ is split,  with split partition $C\uplus I$, where
$C=\{\{2\}\}$ and $I=\{\{3\}, \{5\}, \{7\}, \{13\} \}$.  Suppose now that $q> 3$.
In this case,  the corresponding diagram of the compact form ${\rm GK_c}(L)$ is shown in Figure 4  (see \cite[Lemmas 2.1 and 2.3 ]{Staroletov}).

\vspace{0.5cm} {  \setlength{\unitlength}{4.5mm}
\begin{picture}(0,0)(-13,3)
\linethickness{0.3pt} %
\put(0,0){\circle*{0.4}}
\put(5,0){\circle*{0.4}}
\put(2.5,2){\circle*{0.4}}
\put(2.5,3.8){\circle*{0.4}}
\put(-2.5,2){\circle*{0.4}}
\put(7.5,2){\circle*{0.4}}
\put(4.25,-1){ $R_1$}%
\put(1.75,4.6){ $\{p\}$}%
\put(1.75,0.8){ $R_4$}%
\put(-3.25,2.5){ $R_6$}%
\put(-0.75,-1){ $R_2$}%
\put(6.75,2.5){ $R_3$}%
\put(0,0){\line(1,0){5}}
\put(0,0){\line(5,4){2.5}}
\put(0,0){\line(2,3){2.5}}
\put(5,0){\line(-5,4){2.5}}
\put(5,0){\line(-2,3){2.5}}
\put(2.5,2){\line(0,1){1.8}}
\put(5,0){\line(5,4){2.5}}
\put(0,0){\line(-5,4){2.5}}
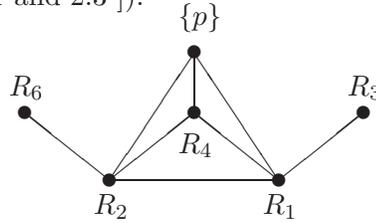
\put(-4,-3){ {\bf Fig. 4}  ${\rm GK_c}(B_3 (q))={\rm GK_c}(C_3 (q))$.}
\end{picture}}\vspace{3cm}

Therefore, the compact form ${\rm GK_c}(L)$ is split, with special split partition $C\uplus I$, where
$C=\{\{p\}, R_1, R_2, R_4\}$ and $I=\{R_3, R_6\}$.

\noindent {\em Type  $G_2$}  \ \  Let $L=G_2(q)$, $q=p^m>2$.
The compact form ${\rm GK_c}(L)$ is shown in Figures 5-7 (see \cite{11VasVd.t}).
If $q\equiv 0\pmod{3}$, then ${\rm GK_c}(L)$ is split, with special split partition $C\uplus I$, where
$$C=\{R\},  \ I=\{R_3, R_6\};  \ \ \ \  C=\{R_3\}, \ I=\{R, R_6\};  \ \ \ \mbox{or}  \ \ \  C=\{R_6\}, \ I=\{R, R_3\};$$
while if $q\equiv 1\pmod{3}$ (resp.  $q\equiv -1\pmod{3}$), then ${\rm GK_c}(L)$ is split, with special split partition $C\uplus I$, where
$$\begin{array}{lll}
C=\{R_1, \{3\}\},  \ \  I=\{R_2\cup \{p\}, R_3, R_6\} & \mbox{if} &  q\equiv 1\pmod{3} \\[0.3cm]
C=\{R_2, \{3\}\}, \ \  I=\{R_1\cup \{p\}, R_3, R_6\} & \mbox{if} &  q\equiv -1\pmod{3}.\\
\end{array}$$

\vspace{1.3cm} {  \setlength{\unitlength}{4.5mm}
\begin{picture}(0,0)(-1,0)
\linethickness{0.3pt} %
\put(0,0){\circle*{0.4}}
\put(0,2){\circle*{0.4}}
\put(0,-2){\circle*{0.4}}
\put(10.5,0){\circle*{0.4}}
\put(12,2){\circle*{0.4}}
\put(13.5,0){\circle*{0.4}}
\put(16.5,0){\circle*{0.4}}
\put(15,2){\circle*{0.4}}
\put(25,0){\circle*{0.4}}
\put(22,0){\circle*{0.4}}
\put(23.5,2){\circle*{0.4}}
\put(28,0){\circle*{0.4}}
\put(26.5,2){\circle*{0.4}}
\put(0.3,1.7){ $R_1\cup R_2\cup \{3\}$}%
\put(0.3,-0.3){ $R_3$}%
\put(0.5,-2.3){$R_6$}%
\put(8.5,-1.5){ $R_2\cup \{p\}$}%
\put(10,1.75){ $R_1$}%
\put(16,-1.5){ $R_3$}
\put(15.2,1.75){ $\{3\}$}
\put(12.8, -1.5){ $R_6$}
\put(24.5,-1.5){ $R_3$}
\put(20,-1.5){ $R_1\cup \{p\}$}
\put(21.5,1.75){ $R_2$}
\put(27.5,-1.5){ $R_6$}
\put(26.8,1.75){ $\{3\}$}
\put(10.5,0){\line(3,4){1.5}}
\put(12,2){\line(1,0){3}}
\put(15,2){\line(3,-4){1.6}}
\put(28,0){\line(-3,4){1.5}}
\put(23.5,2){\line(1,0){3}}
\put(22,0){\line(3,4){1.6}}
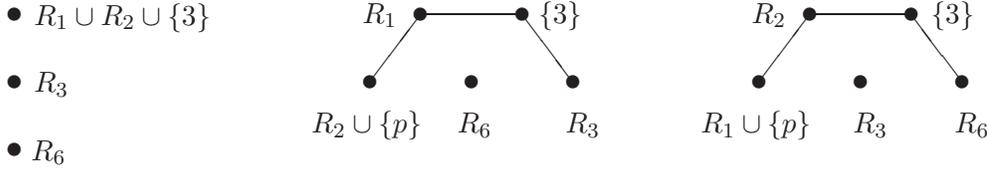
\put(-2,-4){ {\bf Fig. 5}  ${\rm GK_c}(G_2(q))$, $3|q$}
\put(8,-4){ {\bf Fig. 6}  ${\rm GK_c}(G_2(q))$, $3|(q-1)$}
\put(20,-4){ {\bf Fig.  7}  ${\rm GK_c}(G_2(q))$, $3|(q+1)$}
\end{picture}}\vspace{2.3cm}

 {\em Type  $F_4$}  \ \  Let $L=F_4(q)$, $q=p^m$. The graph ${\rm GK}_c(L)$ is depicted in Figures 8  and  9 according to $q$ is even or odd,
respectively  (see \cite{11VasVd.t}). Here,   $R=R_1\cup R_2\cup \{2\}$ and $R'=R_1\cup R_2\cup \{p\}$
$(p\neq 2)$.  The compact form ${\rm GK_c}(L)$ indicates that $\{R, R_3\}$ is a clique if $q=2^m>2$, while
$\{\{2\}, R,  R_3\}$ is a clique if $q=p^m$, $p\neq 2$, and the remaining vertices are pairwise nonadjacent.
 Thus,
 ${\rm GK_c}(L)$ is a split graph with a special split partition $C\uplus I$, where
$C=\{R, R_3\}$ and $I=\{R_4, R_6, R_8, R_{12}\}$, if  $q=2^m>2,$
 and
$C=\{\{2\}, R, R_3\} $ and  $I=\{R_4, R_6, R_8, R_{12}\}$,  if  $q=p^m$, $p\neq 2$.

\vspace{1.2cm} {  \setlength{\unitlength}{4.5mm}
\begin{picture}(0,0)(-11,0)
\linethickness{0.3pt} %
\put(0,1.5){\circle*{0.4}}
\put(0,-1.5){\circle*{0.4}}
\put(-7,0){\circle*{0.4}}
\put(-3.5,-1.5){\circle*{0.4}}
\put(-5,0){\circle*{0.4}}
\put(-3.5,1.5){\circle*{0.4}}
\put(-5.8,0.4){ $R$}%
\put(-3.3,-1.6){ $R_6$}%
\put(-8.5,0){$R_3$}%
\put(0.3,-1.6){ $R_{12}$}%
\put(0.3,1.35){ $R_{8}$}%
\put(-3.1,1.35){$R_4$}%
\put(-5,0){\line(1,1){1.5}}
\put(-7,0){\line(1,0){2}}
\put(-5,0){\line(1,-1){1.5}}
\put(-9.5,-4){ {\bf Fig. 8}  ${\rm GK_c}(F_4(q))$, $q=2^m>2$,}
\put(16.5,-1.5){\circle*{0.4}}
\put(13.5,1.5){\circle*{0.4}}
\put(13.5,-1.5){\circle*{0.4}}
\put(11.3,1.5){\circle*{0.4}}
\put(16.5,1.5){\circle*{0.4}}
\put(11.3,-1.5){\circle*{0.4}}
\put(9.1,0){\circle*{0.4}}
\put(13.8,1.4){ $R_4$}%
\put(16.8,-1.7){ $R_{12}$}%
\put(11,1.9){$R'$}%
\put(13.8,-1.7){ $R_6$}%
\put(17,1.3){$R_8$}%
\put(10.7,-2.6){$\{2\}$}%
\put(7.6,-0.1){$R_3$}%
\put(11.3,1.5){\line(1,0){2.2}}
\put(11.3,1.5){\line(3,-4){2.27}}
\put(11.3,-1.5){\line(1,0){2.27}}
\put(11.3,-1.5){\line(3,4){2.27}}
\put(9.1,0){\line(3,2){2.27}}
\put(11.3,-1.5){\line(0,1){3}}
\put(11.3,-1.5){\line(-3,2){2.27}}
\put(6.5,-4){ {\bf Fig. 9}  ${\rm GK_c}(F_4(q))$, $q=p^m$, $p\neq 2$,}
\end{picture}}\\[2cm]

\noindent {\em Type  $E_6^\epsilon$} \ \
Suppose $L=E_6^\epsilon(q)$. The compact form  ${\rm GK}_c(L)$ is depicted in Figure  10.
Here, the set $\{\{3\}, \{p\}, R_1, R_2, R_{\nu_{\epsilon}}(3), R_{\nu_{\epsilon}}(6)\}$ forms a clique, and the remaining vertices are pairwise nonadjacent.  Thus
 ${\rm GK_c}(L)$ is a split graph with a special split partition $C\uplus I$, where
$C=\{\{3\}, \{p\}, R_1, R_2, R_{\nu_{\epsilon}}(3), R_{\nu_{\epsilon}}(6)\}$ and $I=\{R_4, R_8, R_{12}, R_{\nu_{\epsilon}}(5), R_{\nu_{\epsilon}}(9)\}$.

\vspace{1cm} {  \setlength{\unitlength}{4.5mm}
\begin{picture}(0,0)(-12.5,8)
\linethickness{0.3pt} %
\put(0,-2){\circle*{0.4}}%
\put(6,-2){\circle*{0.4}}%
\put(0,8){\circle*{0.4}}%
\put(6,8){\circle*{0.4}}%
\put(-2,0){\circle*{0.4}}%
\put(-2,6){\circle*{0.4}}%
\put(8,0){\circle*{0.4}}%
\put(8,6){\circle*{0.4}}%
\put(-3,3){\circle*{0.4}}%
\put(9,3){\circle*{0.4}}%
\put(3,-3){\circle*{0.4}}%
\put(-1.2,-3){ $R_2$}%
\put(-4.3,-0.2){ $\{p\}$}%
\put(-5.3,6){ $R_{\nu_\epsilon(6)}$}%
\put(-0.5,8.7){ $R_4$}%
\put(5,8.7){ $R_{\nu_\epsilon(5)}$}%
\put(8.2,6){ $R_{\nu_\epsilon(3)}$}%
\put(8.2,-0.2){ $\{3\}$}%
\put(5.5,-3){ $R_1$}%
\put(-6.3,3){ $R_{\nu_\epsilon(9)}$}%
\put(9.2,2.75){ $R_{12}$}%
\put(2.2,-4){ $R_{8}$}%
\put(0,-2){\line(1,0){6}}
\put(0,-2){\line(4,1){8}}
\put(0,-2){\line(1,1){8}}
\put(0,-2){\line(3,5){6}}
\put(0,-2){\line(0,1){10}}
\put(0,-2){\line(-1,1){2}}
\put(0,-2){\line(-1,4){2}}
\put(6,-2){\line(1,1){2}}
\put(6,-2){\line(1,4){2}}
\put(6,-2){\line(0,1){10}}
\put(6,-2){\line(-3,5){6}}
\put(6,-2){\line(-1,1){8}}
\put(6,-2){\line(-4,1){8}}
\put(8,-0){\line(0,1){6}}
\put(8,-0){\line(-1,0){10}}
\put(8,-0){\line(-1,4){2}}
\put(8,-0){\line(-1,1){8}}
\put(8,-0){\line(-5,3){10}}
\put(-2,0){\line(1,0){6}}
\put(-2,0){\line(0,1){6}}
\put(-2,0){\line(1,4){2}}
\put(-2,0){\line(1,1){8}}
\put(-2,0){\line(5,3){10}}
\put(-2,0){\line(1,0){6}}
\put(-2,0){\line(1,0){6}}
\put(-2,0){\line(1,0){6}}
\put(8,6){\line(1,-3){1}}
\put(3,-3){\line(3,1){3}}
\put(3,-3){\line(-3,1){3}}
\put(3,-3){\line(5,3){1}}
\put(4.2,-2.3){\line(5,3){1}}
\put(5.5,-1.5){\line(5,3){1}}
\put(6.75,-0.75){\line(5,3){1}}
\put(-2,6){\line(1,0){10}}
\put(-2,6){\line(1,1){2}}
\put(-1,-5.5){ {\bf Fig. 10}  ${\rm GK_c}(E_6^\epsilon(q))$.}
\put(8,-1.5){$(q-\epsilon 1)_3\neq 3$ and $p\neq 3$.}
\end{picture}}\vspace{6.5cm}

\noindent {\em Type  $E_7$} \ \
Suppose $L=E_7(q)$. The compact form ${\rm GK}_c(L)$ is depicted in Figure 11.
Here, the set $\{\{p\}, R_1, R_2, R_3, R_4, R_6\}$ forms a clique, and the remaining vertices are pairwise nonadjacent.
Thus, the compact form
 ${\rm GK_c}(L)$ is a split graph with a special split partition $C\uplus I$, where
$C=\{\{p\}, R_1, R_2, R_3, R_4, R_6\}$  and $I=\{R_5, R_7, R_8, R_9, R_{10}, R_{12}, R_{14}, R_{18}\}$.

\vspace{3cm} {  \setlength{\unitlength}{4.5mm}
\begin{picture}(0,0)(-12.2,5)
\linethickness{0.4pt} %
\put(0,0){\circle*{0.4}}%
\put(6,0){\circle*{0.4}}%
\put(0,10){\circle*{0.4}}%
\put(1,3){\circle*{0.4}}%
\put(3,-1){\circle*{0.4}}%
\put(6,10){\circle*{0.4}}%
\put(-2,8){\circle*{0.4}}%
\put(8,8){\circle*{0.4}}%
\put(-2,2){\circle*{0.4}}%
\put(8,2){\circle*{0.4}}%
\put(-2,0){\circle*{0.4}}%
\put(0,-2){\circle*{0.4}}%
\put(6,-2){\circle*{0.4}}%
\put(8,0){\circle*{0.4}}%
\put(-3.8,1.8){$\{p\}$}%
\put(-3.8,7.8){ $R_6$}%
\put(-0.7,10.5){ $R_{12}$}%
\put(0.9,3.3){\tiny $R_{10}$}%
\put(2.3,-2){ $R_{5}$}%
\put(5.5,10.5){ $R_{8}$}%
\put(8.2,7.8){ $R_{4}$}%
\put(8.2,1.8){ $R_3$}%
\put(-1.4,-0.8){ $R_1$}%
\put(6,-0.8){ $R_2$}%
\put(-3.8,-0.2){ $R_7$}%
\put(0.2,-2.2){ $R_9$}%
\put(4.1,-2.2){ $R_{14}$}%
\put(8.2,-0.2){ $R_{18}$}%
\put(0,0){\line(1,0){6}}
\put(0,0){\line(3,-1){3}}
\put(0,0){\line(1,3){1}}
\put(0,0){\line(0,1){10}}
\put(0,0){\line(4,1){8}}
\put(0,0){\line(3,5){6}}
\put(0,0){\line(1,1){8}}
\put(0,0){\line(-1,1){2}}
\put(0,0){\line(-1,4){2}}
\put(0,0){\line(-1,0){2}}
\put(0,0){\line(0,-1){2}}
\put(6,0){\line(0,1){10}}
\put(6,0){\line(1,1){2}}
\put(6,0){\line(1,4){2}}
\put(6,0){\line(-1,1){8}}
\put(6,0){\line(-3,5){6}}
\put(6,0){\line(-5,3){5}}
\put(6,0){\line(-3,-1){3}}
\put(6,0){\line(-4,1){8}}
\put(6,0){\line(0,-1){2}}
\put(6,0){\line(1,0){2}}
\put(-2,2){\line(1,0){10}}
\put(-2,2){\line(0,1){6}}
\put(-2,2){\line(1,4){2}}
\put(-2,2){\line(5,-3){5}}
\put(-2,2){\line(3,1){3}}
\put(-2,2){\line(5,3){10}}
\put(-2,2){\line(1,1){8}}
\put(8,2){\line(0,1){6}}
\put(8,2){\line(-1,1){8}}
\put(8,2){\line(-5,3){10}}
\put(8,2){\line(-5,-3){5}}
\put(8,2){\line(0,1){4}}
\put(-2,8){\line(1,0){10}}
\put(-2,8){\line(1,1){2}}
\put(-2,8){\line(3,-5){3}}
\put(8,8){\line(-1,1){2}}
\put(-0.5,-4.5){ {\bf Fig. 11}  ${\rm GK_c}(E_7(q))$.}
\end{picture}}\vspace{5cm}

\noindent {\em Type  $E_8$} \ \
Suppose $L=E_8(q)$. The compact form ${\rm GK}_c(L)$ is depicted in Figure 12. Here, the vector from $5$ to $R_4$ and the dotted edge $\{5,R_{20}\}$ indicate that $R_4$ and $R_{20}$ are not connected, but if $5\in R_4$ (i.e., $q^2\equiv  -1\pmod{5}$), then there exists an edge joining $5$ and $R_{20}$. Now $\{R , R_3 , R_4 , R_6\}$ forms a clique, and the remaining vertices are pairwise nonadjacent. Thus
 ${\rm GK_c}(L)$ is a split graph with a special split partition $C\uplus I$, where
$C=\{R , R_3 , R_4 , R_6\}$  and $I=\{R_5, R_8, R_9, R_{10}, R_{12}, R_{15}, R_{18}, R_{20}, R_{24}, R_{30}\}$.

\vspace{0.5cm} {  \setlength{\unitlength}{4.5mm}
\begin{picture}(0,0)(-15.5,8)
\linethickness{0.4pt} %
\put(0,0){\circle*{0.4}}%
\put(0,5){\circle*{0.4}}%
\put(-2,8){\circle*{0.4}}%
\put(2,8){\circle*{0.4}}%
\put(-4,4){\circle*{0.4}}%
\put(4,4){\circle*{0.4}}%
\put(-6,7){\circle*{0.4}}%
\put(6,7){\circle*{0.4}}%
\put(-6,-1){\circle*{0.4}}%
\put(6,-1){\circle*{0.4}}%
\put(0,-2){\circle*{0.4}}%
\put(0,-4){\circle*{0.4}}%
\put(0,-5){\circle*{0.4}}%
\put(-3,-5){\circle*{0.4}}%
\put(3,-5){\circle*{0.4}}%
\put(0,5){\line(0,-1){5}}
\put(0,5){\line(1,-1){6}}
\put(0,5){\line(4,-1){4}}
\put(0,5){\line(3,1){6}}
\put(0,5){\line(2,3){2}}
\put(0,5){\line(-2,3){2}}
\put(0,5){\line(-3,1){6}}
\put(0,5){\line(-4,-1){4}}
\put(0,5){\line(-1,-1){6}}
\put(0,0){\line(6,-1){6}}
\put(0,0){\line(-6,-1){6}}
\put(0,0){\line(1,1){4}}
\put(0,0){\line(-1,1){4}}
\put(0,0){\line(1,4){2}}
\put(0,0){\line(-1,4){2}}
\put(6,-1){\line(-2,5){2}}
\put(-6,-1){\line(2,5){2}}
\put(4,4){\line(2,3){2}}
\put(4,4){\line(-1,2){2}}
\put(4,4){\line(-3,2){6}}
\put(-4,4){\line(-2,3){2}}
\put(-4,4){\line(1,2){2}}
\put(-4,4){\line(3,2){6}}
\put(-4,4){\line(1,0){8}}
\put(0,-2){\vector(0,1){1.8}}
\put(-0.1,-3.3){$\vdots$}
\put(-0.30,5.6){$R$}%
\put(0.3,-0.9){$R_{4}$}%
\put(0.4,-2.2){$\{5\}$}%
\put(0.4,-4.2){$R_{20}$}%
\put(6.3,-1){$R_{10}$}%
\put(-7.4,-1){$R_{5}$}%
\put(4.4,4){$R_{6}$}%
\put(-5.4,4){$R_{3}$}%
\put(5.4,7.5){$R_{18}$}%
\put(-6.7,7.5){$R_{9}$}%
\put(1.5,8.5){$R_{8}$}%
\put(-2.6,8.5){$R_{12}$}%
\put(0.4,-5.2){$R_{24}$}%
\put(-2.6,-5.2){$R_{15}$}%
\put(3.4,-5.2){$R_{30}$}%
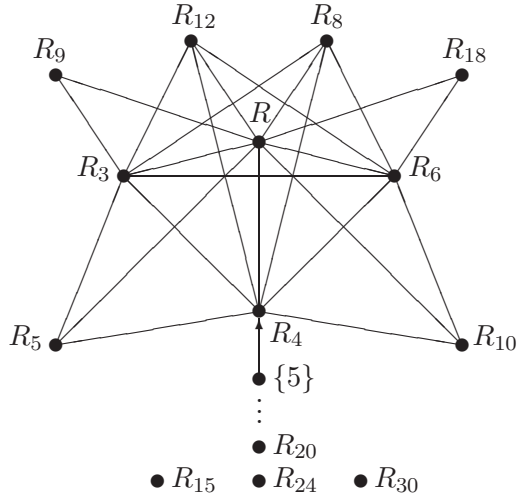
\put(-7,-7){ {\bf Fig. 12}  ${\rm GK_c}(E_8(q))$, $R=R_1\cup R_2\cup \{p\}$.}
\end{picture}}\vspace{7.5cm}

\noindent {\em Type  ${^2B}_2$} \ \
Let $L={^2B}_2(q)$, $q=2^{2n+1}$ $(n\geqslant
1)$. By \cite{shi}, we know that $$\mu(L)=\left\{4, \
q-1, \ q-\sqrt{2q}+1, \ q+\sqrt{2q}+1\right\}.$$
The compact form ${\rm GK_c}(L)$ is depicted in Figure 13.

\vspace{0.3cm} {  \setlength{\unitlength}{4.5mm}
\begin{picture}(0,0)(-10,0)
\linethickness{0.3pt} %
\put(13,0){\circle*{0.4}}
\put(5,0){\circle*{0.4}}
\put(-6,0){\circle*{0.4}}
\put(-1,0){\circle*{0.4}}
\put(-6.2,-1.5){$\{2\}$}%
\put(-2.5,-1.5){$\pi(q-1)$}%
\put(2.5,-1.5){$\pi(q-\sqrt{2q}+1)$}%
\put(10.5,-1.5){$\pi(q+\sqrt{2q}+1)$}%
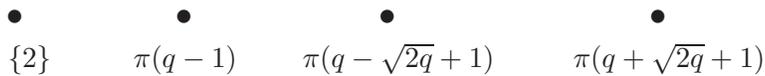
\put(-3,-3.5){ {\bf Fig. 13}  \  ${\rm GK_c}({^2B}_2(q))$, $q=2^{2m+1}>2.$}
\end{picture}}\vspace{2.2cm}

Thus
 ${\rm GK_c}(L)$ is a split graph with a special split partition $C\uplus I$, where
 $$\begin{array}{l}
C=\{\{2\}\}   \ \mbox{and} \  I=\{\pi(q-1), \pi(q-\sqrt{2q}+1), \pi(q+\sqrt{2q}+1) \},\\[0.3cm]
C=\{\pi(q-1)\}  \ \mbox{and} \   I=\{\{2\}, \pi(q-\sqrt{2q}+1), \pi(q+\sqrt{2q}+1) \},\\[0.3cm]
C=\{\pi(q-\sqrt{2q}+1)\} \  \mbox{and} \  I=\{\{2\}, \pi(q-1), \pi(q+\sqrt{2q}+1) \}, \ \mbox{or}\\[0.3cm]
C=\{\pi(q+\sqrt{2q}+1)\} \  \mbox{and} \  I=\{\{2\}, \pi(q-1), \pi(q-\sqrt{2q}+1)\}.
\end{array}$$

\noindent {\em Type  ${^3D}_4$} \ \
Suppose  $L={^3D}_4(q)$, $q=p^m$. The compact form  ${\rm GK_c}(L)$ is shown in Figure 14.
Therefore, if $R_6\neq \emptyset$, then ${\rm GK_c}(L)$ is split, with special split partition $C\uplus I$, where
$$C=\{R, R_3\}, \ \ I=\{R_6, R_{12}\} \ \ \ \ \ \mbox{or} \ \ \ \ \ \  C=\{R, R_{6}\}, \ \ I=\{R_3, R_{12}\}. $$
When $R_6=\emptyset$, we have the special split partition $C\uplus I$, where $C=\{R, R_3\}$ and $I=\{R_{12}\}$.

\vspace{1cm} {  \setlength{\unitlength}{4.5mm}
\begin{picture}(0,0)(-15.5,0)
\linethickness{0.3pt} %
\put(3,0){\circle*{0.4}}
\put(-3,0){\circle*{0.4}}
\put(-6,0){\circle*{0.4}}
\put(0,0){\circle*{0.4}}
\put(-6.8,0.5){ $R_3$}%
\put(-0.8,0.5){ $R_6$}%
\put(-3.3,0.5){$R$}%
\put(2,0.5){ $R_{12}$}%
\put(-6,0){\line(1,0){6}}
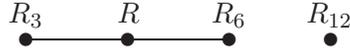
\put(-9.5,-2){ {\bf Fig. 14}  \  ${\rm GK_c}({^3D}_4(q))$, $R=R_1\cup R_2\cup \{p\}$.}
\end{picture}}\vspace{1.5cm}

\noindent {\em Type  ${^2G}_2$} \ \
Let $L$ be the simple group ${^2G}_2(q)$, $q=3^{2m+1}$, $m\geqslant
1$. Then by Lemma 4 in \cite{brandl} (see also \cite[Sec. XI.13]{Huppert}), we have
$$\mu(L)=\{6, 9, q-1, (q+1)/2, q-\sqrt{3q}+1,
q+\sqrt{3q}+1\}.$$
In particular, $$\mu_2 (L) = \{q-\sqrt{3q}+1\} \ \ \ \mbox{and} \ \ \ \mu_3(L) = \{q+\sqrt{3q}+1\}.$$
The compact form ${\rm GK_c}(L)$ is depicted in Figure 15. Thus
 ${\rm GK_c}(L)$ is a split graph with a special split partition $C\uplus I$, where
$$C=\{\{2\}, \{3\}, \pi((q-1)/2)\} \ \ \mbox{and} \ \ I=\{\pi((q+1)/4), \pi(q-\sqrt{3q}+1), \pi(q+\sqrt{3q}+1) \}, \ \mbox{or}$$
$$C=\{\{2\}, \{3\}, \pi((q+1)/4)\} \ \ \mbox{and} \ \ I=\{\pi((q-1)/2), \pi(q-\sqrt{3q}+1), \pi(q+\sqrt{3q}+1) \}.$$

\vspace{1cm} {  \setlength{\unitlength}{4.5mm}
\begin{picture}(0,0)(-11,0)
\linethickness{0.3pt} %
\put(7,1.5){\circle*{0.4}}
\put(7,-1.5){\circle*{0.4}}
\put(0,1.5){\circle*{0.4}}
\put(0,-1.5){\circle*{0.4}}
\put(-5,0){\circle*{0.4}}
\put(-2.3,0){\circle*{0.4}}
\put(0.5,1.3){$\pi((q-1)/2)$}%
\put(0.5,-1.7){$\pi((q+1)/4)$}%
\put(-5.5,0.6){$\{3\}$}%
\put(-3.4,0.6){$\{2\}$}%
\put(7.8,-1.7){$\pi(q-\sqrt{3q}+1)$}%
\put(7.8,1.3){$\pi(q+\sqrt{3q}+1)$}%
\put(0,1.5){\line(-3,-2){2.27}}
\put(0,-1.5){\line(-3,2){2.27}}
\put(-5,0){\line(1,0){2.7}}
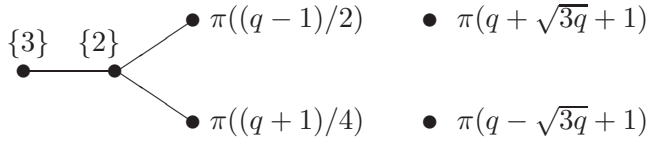
\put(-2,-4){ {\bf Fig. 15}  \  ${\rm GK_c}({^2G}_2(q))$, $q=3^{2m+1}>3.$}
\end{picture}}\vspace{2.5cm}

\noindent {\em Type  ${^2F}_4$} \ \
Let $L={^2F}_{4}(q)$, where $q=2^{2m+1}$.
Note that, this group is nonabelian simple unless $m=0$. In this case, the group ${^2F}_4(2)'$ is simple and is called the Tits group. The set $\omega(L)$ is exactly the set of
all divisors of the following numbers (see \cite{Deng-Shi}):
\begin{itemize}
\item[{\rm (1)}] $12$, $16$,
$2(q+1)$, $4(q-1)$, $4(q\pm \sqrt{2q}+1), q^2\pm 1$, $q^2-q+1$,
$(q+1)(q\pm\sqrt{2q}+1)$;

\item[{\rm (2)}] $q^2-\sqrt{2q^3}+q-\sqrt{2q}+1$;

\item[{\rm (3)}]  $q^2+\sqrt{2q^3}+q+\sqrt{2q}+1$.
\end{itemize}
Note that $\mu_2 (L) = \{q^2+\sqrt{2q^3}+q+\sqrt{2q}+1\}$ and $\mu_3(L) = \{q^2+\sqrt{2q^3}+q+\sqrt{2q}+1\}$. The compact form ${\rm GK_c}(L)$ is depicted in Figure  16, where  $U=\pi(q^2-q+1)\setminus \{3\}$, $V=\pi(q-\sqrt{2q}+1)$. Now $\{\{2\}, \{3\} , V, \pi(q+\sqrt{2q}+1), \pi(q+1)\setminus \{3\}\}$ forms a clique, and the remaining vertices are pairwise nonadjacent. Thus
 ${\rm GK_c}(L)$ is a split graph with a special split partition $C\uplus I$, where
$$C=\{\{2\}, \{3\} , V, \pi(q+\sqrt{2q}+1), \pi(q+1)\setminus \{3\}\} ,$$ and
$$I=\{U, \pi(q-1), \pi(q^2-\sqrt{2q^3}+q-\sqrt{2q}+1), \pi(q^2+\sqrt{2q^3}+q+\sqrt{2q}+1)\}.$$
If $L={^2F}_4(2)'$, then we have $\mu(L)=\{12, 13, 16, 20\}$. Therefore, ${\rm GK_c}(L)={\rm GK}(L)$
 is a split graph with a special split partition $C\uplus I$ where $C=\{2, 3\}$ and $I=\{5, 13\}$, or $C=\{2, 5\}$ and $I=\{3, 13\}$.

{\setlength{\unitlength}{4.5mm}
\begin{picture}(0,0)(-7,5)
\linethickness{0.4pt} %
\put(-2,0){\circle*{0.4}}
\put(0,0){\circle*{0.4}}
\put(8,0){\circle*{0.4}}
\put(5,1){\circle*{0.4}}
\put(5,-1){\circle*{0.4}}
\put(8,4){\circle*{0.4}}
\put(8,-4){\circle*{0.4}}
\put(12,2){\circle*{0.4}}
\put(12,-2){\circle*{0.4}}
\put(-3,0){$U$}
\put(-0.8,0.5){$\{3\}$}
\put(4.1,-1.9){$\{2\}$}
\put(4.5,1.5){$V$}
\put(8.5,3.8){$\pi(q+\sqrt{2q}+1)$}
\put(8.5,-0.2){$\pi(q+1)\setminus \{3\}$}
\put(8.5,-4.2){$\pi(q-1)$}
\put(12.7,-2.2){$\pi(q^2-\sqrt{2q^3}+q-\sqrt{2q}+1)$}
\put(12.7,1.8){$\pi(q^2+\sqrt{2q^3}+q+\sqrt{2q}+1)$}
\put(-2,0){\line(1,0){10}}
\put(0,0){\line(5,1){5}}
\put(0,0){\line(5,-1){5}}
\put(0,0){\line(2,1){8}}
\put(0,0){\line(2,-1){8}}
\put(5,1){\line(1,1){3}}
\put(5,1){\line(3,-1){3}}
\put(5,1){\line(0,-1){2}}
\put(5,-1){\line(3,5){3}}
\put(5,-1){\line(3,1){3}}
\put(5,-1){\line(1,-1){3}}
\put(8,4){\line(0,-1){8}}
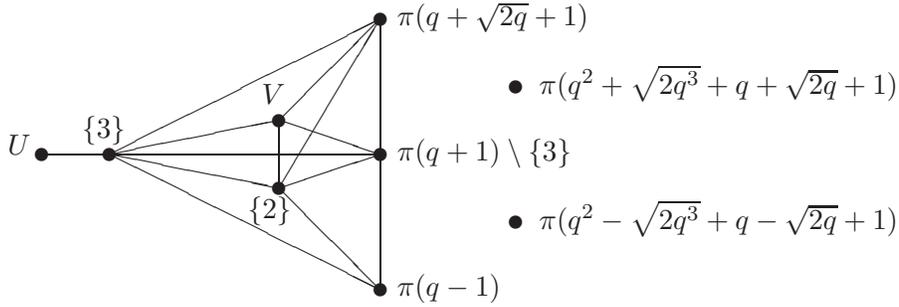
\put(4.25,-6.5){ {\bf Fig. 16}  ${\rm GK_c}({^2F}_{4}(q))$, $q=2^{2m+1}$.}
\end{picture}}\vspace{6cm}

We have thus examined all simple groups of Lie type, and Theorem C is proved.



\section{\sc Examples of Nonsplitness}
\subsection{\sc Nonsplitness of ${\rm GK}(L)$ and $\mathcal{S}(L)$}

As pointed out in the introduction, the prime graph of a finite simple group of Lie type is generally nonsplit. For brevity, we will illustrate this only in the case of linear groups. But the same reasoning works for other classical families as well.

\begin{proposition}\label{nonspl-GK}
Let $L=A_{n-1}(q)={\rm PSL}_{n}(q)$ be a finite simple linear group, where $q$ is a power of a prime $p$. Assume that $q>p$ and $n>11$. Then the prime graph ${\rm GK}(L)$ is nonsplit.
\end{proposition}
\begin{proof}
Let $q=p^a$, where $a\geqslant 2$. Let $k>1$ be an integer and denote $a'=(a)_{\pi(k)}$. By Theorem~\ref{zsig}, if at least one of the sets $R_{ka}(p)$ and $R_{ka'}(p)$ is empty, then $p=2$, and either $k=2$ and $a=3$, or $k=3$ and $a=2$, or $k=6$ and $a'=1$.
Since $n>11$, we have that $R_k(p)$ is nonempty for any integer $k\in (\frac{n}{2},n)$.
Moreover, the interval $(\frac{n}{2},n)$ contains at least two numbers $k_1$ and $k_2$, such that $\pi(a)\nsubseteq\pi(k_i)$. 
It follows from Lemma \ref{divisors} that $|R_{k_1}(q)|>1$ and $|R_{k_2}(q)|>1$. The sets $R_{k_1}(q)$ and $R_{k_2}(q)$ are cliques in ${\rm GK}(L)$. On the other hand, any two numbers $r\in R_{k_1}(q)$ and $s\in R_{k_2}(q)$ are non-adjacent in ${\rm GK}(L)$ by Lemma~\ref{vasi-algebra}(iii), since $k_1+k_2>n$ and $k_1$, $k_2$ do not divide each other. Let $r_1,r_2\in R_{k_1}(q)$ and $s_1,s_2\in R_{k_2}(q)$. Then the vertices $r_1,r_2,s_1,s_2$ induce the subgraph isomorphic to $2K_2$, and hence ${\rm GK}(L)$ is not a split graph by Proposition \ref{forbidden}.
\end{proof}

The idea of the foregoing proof was to find two sets $R_i(q)$ and $R_j(q)$ such that any two vertices $r\in R_i(q)$ and $s\in R_j(q)$ lie in a maximal coclique of ${\rm GK}(L)$. This logic is also applicable for a solvable graph. Recall, however, that the prime graph of a group is a subgraph of the solvable graph of this group, so the situation here is more sophisticated. Our next example requires more strict conditions on a group (but one can construct other infinite series of examples in the same way).

\begin{proposition}
Let $L=A_{n-1}(q)={\rm PSL}_{n}(q)$ be a finite simple linear group, where $q=p^3$. Assume that $n=uw$, where $u$ and $w$ are distinct odd primes such that $n\equiv -1\pmod{3}$. Then the solvable graph $\mathcal{S}(L)$ is nonsplit.
\end{proposition}
\begin{proof}
We have that $n$ and $n-1$ are not divisible by $3$. So, $|R_n(q)|>1$ and $|R_{n-1}(q)|>1$ by Lemma~\ref{divisors}, and each of these two sets is a clique in ${\rm GK}(L)$, and hence in $\mathcal{S}(L)$. We claim that any two pairs $r_1,r_2\in R_n(q)$ and $s_1,s_2\in R_{n-1}(q)$ induce the subgraph isomorphic to $2K_2$. Assume the contrary. Then $L$ contains a solvable subgroup $H$, whose order is divisible by $rs$, where $r\in R_n(q)$ and $s\in R_{n-1}(q)$. It follows from Lemma \ref{AK31} that $s$ divides $n$, and hence either $s=u$ or $s=w$. But $s$ is a primitive prime divisor of $q^{n-1}-1$, and $s$ divides $q^{s-1}-1$ by Fermat's little theorem, so $n-1$ divides $s-1$. In particular, $n\leqslant s$, which is a contradiction.
\end{proof}

\subsection{\sc Nonsplitness of $\mathcal{S}_{\rm c}(L)$}\label{Sc_nonsplit}

We proved that the graph ${\rm GK_c}(L)$ is split for any finite simple group $L$ of Lie type. It turns out that the same assertion is not true for the graph ${\cal S}_{\rm c}(L)$, as the following example shows.

\medskip

\textbf{\sc Example}.
Let $L=A_{10}(2)={\rm PSL}_{11}(2)$, and let $r=r_{11}(2)$. Then $L$ has a solvable subgroup of order divisible by $11\cdot r$, which can be obtained as the normalizer of an $r$-subgroup of $L$. Since $2$ is a primitive root modulo $11$, we have that $11=r_{10}(2)$. Thus, the vertices $r_{10}(2)$ and $r_{11}(2)$ are adjacent in the solvable graph $\mathcal{S}(L)$. Note that the pair $\{r_n(q),r_{n-1}(q)\}$ cannot be an edge in the prime graph of ${\rm PSL}_n(q)$ for any $n$, $q$ (moreover, generally these two vertices are nonadjacent in the solvable graph too).
The graph ${\cal S}_{\rm c}(L)$ is depicted in Figure 17 (a vertex labelled $R_i$ denotes the set $R_i(2)$).
The vertices $r_3(2)$, $r_7(2)$, $r_{10}(2)$ and $r_{11}(2)$ induce the forbidden subgraph in $\mathcal{S}(L)$, and $R_3$, $R_7$, $R_{10}$ and $R_{11}$ --- in ${\cal S}_{\rm c}(L)$ (for details, see the proof of Proposition \ref{nonspl-Sc}).

\vspace{1cm} {  \setlength{\unitlength}{5.5mm}
\begin{picture}(0,0)(-10,4)
\linethickness{0.3pt} %
\put(0,0){\circle*{0.4}}
\put(6,0){\circle*{0.4}}
\put(6,-2){\circle*{0.4}}
\put(2,-2){\circle*{0.4}}
\put(4,-2){\circle*{0.4}}
\put(0,2){\circle*{0.4}}
\put(6,2){\circle*{0.4}}
\put(2,4){\circle*{0.4}}
\put(4,4){\circle*{0.4}}

\put(2.7,-3){ $\{2\}\cup R_2$}%
\put(6.3,-2.3){ $R_{11}$}%
\put(6.3,-0.3){ $R_{10} $}%
\put(6.3,2.2){ $R_9 $}%
\put(3.6,4.6){ $R_8 $}%
\put(1.2,4.6){ $R_7 $}%
\put(-1.5,2.2){ $R_5 $}%
\put(-1.5,-0.5){ $R_4 $}%
\put(0.9,-2.9){ $R_3$ }      %

\put(4,-2){\line(1,1){2}}
\put(4,-2){\line(1,2){2}}
\put(4,-2){\line(0,1){6}}
\put(4,-2){\line(-1,3){2}}
\put(4,-2){\line(-1,1){4}}
\put(4,-2){\line(-2,1){4}}
\put(4,-2){\line(-1,0){2}}

\put(2,-2){\line(1,1){4}}
\put(2,-2){\line(1,3){2}}
\put(2,-2){\line(0,1){6}}
\put(2,-2){\line(-1,2){2}}
\put(2,-2){\line(-1,1){2}}

\put(0,0){\line(1,0){6}}
\put(0,0){\line(1,1){4}}
\put(0,0){\line(1,2){2}}
\put(0,0){\line(0,1){2}}

\put(0,2){\line(3,-1){6}}
\put(6,0){\line(0,-1){2}}

\put(-1.5,-4.5){ {\bf Fig. 17}  $\mathcal{S}_{\rm c}({\rm PSL}_{11}(2))$.}
\end{picture}}\vspace{5cm}

The key point here is the adjacency of $r_n$ and $r_{n-1}$ (as well as corresponding vertices of ${\cal S}_{\rm c}(L)$), which is possible if and only if $n\in R_{n-1}(2)$. If we could assume that $2$ is a primitive root for infinitely many prime numbers $n$, we would obtain infinitely many examples of nonsplit graphs for groups ${\rm PSL}_n(2)$ arguing as above. In fact, this assumption is the essence of the well-known number-theoretic conjecture of E.~Artin (see~\cite{Artin27}):

\begin{coj*}[E. Artin, 1927] Any integer $k$, other than $-1$ or a perfect square, is a primitive root for infinitely many primes.
\end{coj*}
This conjecture is still unresolved, but partial results on it are numerous. In 1984, Gupta and Ram Murty~\cite{Gupta} showed unconditionally that there are infinitely many $k$ satisfying Artin's conjecture.
Heath-Brown \cite{heath-brown} improved the method used by Gupta and Ram Murty and obtained the following:

\begin{lm}[{\cite[Corollary 2]{heath-brown}}]\label{GK_c_non}
There are at most two positive primes for which Artin's conjecture fails.
\end{lm}

The result of Heath-Brown yields that there are infinitely many pairs of primes $(p,n)$, such that $p$ is a primitive root modulo $n$. It allows us to construct infinitely many examples of nonsplit graphs ${\cal S}_{\rm c}(L)$ for linear groups $L$.

\begin{proposition}\label{nonspl-Sc}
Let $L=A_{n-1}(p)={\rm PSL}_n(p)$, where $p$ and $n$ are primes, $n>13$ and $p$ is a primitive root modulo $n$. Then the graph ${\cal S}_{\rm c}(L)$ is nonsplit.
\end{proposition}
\begin{proof} We will show that ${\cal S}_{\rm c}(L)$ contains a subgraph isomorphic to $2K_2$.
Recall that ${\cal S}_{\rm c}(L)$ is defined as a quotient graph $\mathcal{S}(L)/_\equiv$ w.r.t. the equivalence relation ``$\equiv$'', where two vertices $r$ and $s$ in $\mathcal{S}(L)$ are equivalent if the corresponding balls  $r^\bot$ and $s^\bot$ of radius 1 coincide. For a vertex $r$ of the graph $\mathcal{S}(L)$, denote by $\tilde{r}$ the corresponding equivalence class, i.e., $\tilde{r}$ is a vertex in ${\cal S}_{\rm c}(L)$.
For brevity, set $r_i=r_i(p)$. 
Recall that we write $r\sim s$ if $r$ and $s$ are adjacent in ${\rm GK}(L)$, and $r\approx s$ if $r$ and $s$ are adjacent in $\mathcal{S}(L)$.

The group $L$ contains a solvable subgroup whose order is divisible by $n\cdot r_n$. We have that $p$ is a primitive root modulo $n$, and hence $n=r_{n-1}$, so $r_n\approx r_{n-1}$. On the other hand, $p\approx r_{n-1}$ and $p\not\approx r_n$, so $r_{n-1}\not\equiv r_n$.

Now put $k=\frac{n-5}{2}$ and $m=\frac{n+3}{2}$. Note that $n>13$ is a prime, and $r_i>i$ for every integer $i>2$.
The prime $n$ is not divisible by $k$, $m$, $r_k$ or $r_m$, since $k<n-1$ and $m<n-1$, and $n$ is a primitive prime divisor of $p^{n-1}-1$.
Since $n>13$, we have that
$$
\frac{n-1}{3}<k<\frac{n-1}{2}<m<n-1,
$$
and so $n-1$ is not a multiple of $k$ or $m$. Moreover, $n-1$ is not divisible by $r_k$ and $r_m$. Indeed, assume that $r_m$ divides $n-1$. Since $r_m>m=\frac{n+3}{2}$, the only possibility for $r_m$ is $n-1$. However, $n$ is a prime, so $n-1$ is not. Assume that $r_k$ divides $n-1$. We have that $r_k\geq\frac{n-3}{2}$. If $r_k=\frac{n-3}{2}$, then $\frac{n-1}{3}<r_k<\frac{n-1}{2}$. If $r_k=\frac{n-1}{2}$, then $r_k$ divides $p^{\frac{n-3}{2}}-1$ by Fermat's little theorem. On the other hand, $r_k$ is a primitive prime divisor of $p^{\frac{n-5}{2}}-1$, but $\frac{n-5}{2}$ does not divide $\frac{n-3}{2}$. Finally, $n-1$ is not a prime, so $r_k\neq n-1$.
Therefore, each of the numbers $k$, $m$, $r_k$, $r_m$ does not divide $n$ and $n-1$.

Since $k+m<n$, Lemma~\ref{vasi-algebra}(iii) implies that $r_k\sim r_m$, and so $r_k\approx r_m$.
We claim that $r_k\not\equiv r_m$. Indeed, consider $r_l$, where $l=\frac{n-1}{2}$. Then $r_l\sim r_k$ and $r_l\not\sim r_m$. Assume that $r_l\approx r_m$. Then $L$ has a solvable subgroup $H$, whose order is divisible by $r_l r_m$. Now Lemma \ref{AK31} implies that $r_l$ divides $m$.
But $m=\frac{n+3}{2}$, and $r_l>l=\frac{n-1}{2}$, so the only possibility is that $r_l=\frac{n+3}{2}$. Then $\frac{n+3}{2}$ is a prime and divides $p^{\frac{n+1}{2}}-1$ by Fermat's little theorem, while $\frac{n+3}{2}$ is a primitive prime divisor of $p^{\frac{n-1}{2}}-1$, which is impossible because $\frac{n-1}{2}$ does not divide $\frac{n+3}{2}$.
Thus $r_l\approx r_k$ and $r_l\not\approx r_m$, and hence $r_k\not\equiv r_m$.

It follows from Lemma \ref{AK31} that $L$ does not contain solvable subgroups, whose orders are divisible by $r_m r_n$ or $r_m r_{n-1}$. Indeed, if such a subgroup exists, then $r_m$ divides either $n$ or $n-1$, which is impossible. Therefore, $r_m\not\approx r_n$ and $r_m\not\approx r_{n-1}$.

It remains to show that $r_k\not\approx r_n$ and $r_k\not\approx r_{n-1}$.
If $H$ is a solvable subgroup of $L$ whose order is divisible by $r_n$, then $H$ normalizes its Sylow $r_n$-subgroup by \cite[Lemma 2.5]{13AmKaz}. Therefore, $H$ is contained in the normalizer of a maximal torus of order $\frac{p^n-1}{(p-1)(n,p-1)}$ in $L$, and so the order of $H$ divides $n(p^n-1)$.
Since $r_k$ divides neither $p^n-1$, nor $n$, it cannot divide the order of $H$.
Thus $r_k\not\approx r_n$.
Now assume that $H$ is a solvable subgroup of $L$ whose order is divisible by $r_k r_{n-1}$. Lemma~\ref{tor} implies that $H$ is reducible.
Replacing $n$ by $n-1$, we again conclude that $r_k$ must divide $p^{n-1}-1$ or $n-1$, which is a contradiction. Therefore, $r_k\not\approx r_{n-1}$.

Thus, we have proved that the graph ${\cal S}_{\rm c}(L)$ contains the edges $(\tilde{r}_k,\tilde{r}_m)$ and $(\tilde{r}_{n-1},\tilde{r}_n)$, and the vertices $\tilde{r}_k$, $\tilde{r}_m$, $\tilde{r}_{n-1}$ and $\tilde{r}_n$ induce the subgraph isomorphic to $2K_2$ in ${\cal S}_{\rm c}(L)$.
\end{proof}

\begin{center}
 {\sc Acknowledgments }
\end{center}
The part of this work was done while the third author had a visiting position at the
Department of Mathematical Sciences, Kent State
University, USA. He would like to thank the hospitality of the Department of Mathematical Sciences of KSU.

The forth author was supported by the program of fundamental scientific researches of the SB RAS No. I.1.1., project No. 0314-2016-0001.

\noindent {\sc Mark L. Lewis}\\[0.2cm]
{\sc Department of Mathematical Sciences, Kent State
University,}\\ {\sc  Kent, Ohio $44242$, United States of
America}\\[0.1cm]
{\em E-mail address}: {\tt  lewis@math.kent.edu}\\[0.3cm]
 {\sc J. Mirzajani and A. R. Moghaddamfar}\\[0.2cm]
{\sc Faculty of Mathematics, K. N. Toosi
University of Technology,
 P. O. Box $16315$--$1618$, Tehran, Iran,}\\[0.1cm]
{\em E-mail addresses}:  {\tt  jmirzajani@mail.kntu.ac.ir}, \ \  {\tt moghadam@kntu.ac.ir}\\[0.3cm]
{\sc  A. V. Vasil'ev}\\[0.2cm]
{\sc Sobolev Institute of Mathematics, Novosibirsk, Russia},\\[0.1cm]
{\em E-mail address:} {\tt vasand@math.nsc.ru}\\[0.3cm]
{\sc M. A. Zvezdina}\\[0.2cm]
{\sc Novosibirsk State University, Novosibirsk, Russia},\\[0.1cm]
{\em E-mail address:} {\tt maria.a.zvezdina@gmail.com}\\[0.3cm]

\end{document}